\documentclass[11pt]{amsart}

\usepackage{mathtools}
\usepackage{amsmath, amsthm, amssymb, amsfonts, enumerate}
\usepackage{dsfont}
\usepackage{color}
\usepackage{geometry}
\usepackage{todonotes}
\usepackage{graphicx}
\usepackage{caption}
\usepackage{float}
\usepackage{soul}

\usepackage[colorlinks=true,linkcolor=blue,urlcolor=blue, citecolor=blue]{hyperref}

\mathtoolsset{showonlyrefs}

\def \cA{\mathcal{A}}

\def \cC{\mathcal{C}}

\def \cF{\mathcal{F}}
\def \cG{\mathcal{G}}

\def \cI{\mathcal{I}}
\def \cJ{\mathcal{J}}
\def \cK{\mathcal{K}}
\def \cL{\mathcal{L}}
\def \cM{\mathcal{M}}

\def \cO{\mathcal{O}}

\def \cS{\mathcal{S}}

\def \cT{\mathcal{T}}
\def \P{\mathsf P}

\def \E{\mathsf E}

\def \N{\mathbb{N}}
\def \R{\mathbb{R}}
\def \F{\mathbb F}

\def \ud{\mathrm{d}}
\def \e{\mathrm{e}}

\newcommand{\eps}{\varepsilon}
\newcommand{\ind}{\mathds{1}}

\newtheorem{theorem}{Theorem}[section]
\newtheorem{lemma}[theorem]{Lemma}
\newtheorem{corollary}[theorem]{Corollary}
\newtheorem{proposition}[theorem]{Proposition}

\newtheorem{definition}[theorem]{Definition}
\newtheorem{remark}[theorem]{Remark}
\newtheorem{assumption}[theorem]{Assumption}

\theoremstyle{definition}

\DeclareMathOperator{\sign}{sign}

\DeclareMathOperator{\inter}{int}

\newcommand{\bunderline}[1]{\underline{#1\mkern-4mu}\mkern4mu }
\newcommand{\boverline}[1]{\overline{#1\mkern0mu}\mkern4mu }

\makeatletter
\@namedef{subjclassname@2020}{\textup{2020}Mathematics Subject Classification}
\makeatother

\title[Regularity of the value function in a Stopper vs.\ Singular-Controller Game]{Global regularity of the value function in a \\ Stopper vs.\ Singular-Controller Game}
\author[Bovo]{Andrea Bovo}
\author[Milazzo]{Alessandro Milazzo}
\subjclass[2020]{35R35, 49N60, 60G40, 91A05, 91A15, 93E20}
\keywords{free-boundary problem for PDEs, 2-person games, singular control, optimal stopping}
\address{A.\ Bovo: School of Management and Economics, Dept.\ ESOMAS, University of Torino, Corso Unione Sovietica, 218 Bis, 10134, Torino, Italy.}
\email{\href{mailto:andrea.bovo@unito.it}{andrea.bovo@unito.it}}
\address{A.\ Milazzo: Department of Economics, University of Eastern Piedmont, Via E.\ Perrone, 18, 28100, Novara, Italy.}
\email{\href{mailto:alessandro.milazzo@uniupo.it}{alessandro.milazzo@uniupo.it}}

\date{\today}
\numberwithin{equation}{section}

\begin{document}

\begin{abstract}
We study a class of zero-sum stochastic games between a stopper and a singular-controller, previously considered in \cite{bovo2024saddle}. The underlying singularly-controlled dynamics takes values in $\cO\subseteq\R$. The problem is set on a finite time-horizon and is connected to a parabolic variational inequality of min-max type with spatial-derivative and obstacle constraints.

We show that the value function of the problem is of class $C^1$ in the whole domain $[0,T)\times\cO$ and that the second-order spatial derivative and the second-order mixed derivative are continuous everywhere except for a (potential) jump across a non-decreasing curve (the stopping boundary of the game). The latter discontinuity is a natural consequence of the partial differential equation associated to the problem. Beyond its intrinsic analytical value, such a regularity for the value function is a stepping stone for further exploring the structure and properties of the free-boundaries of the stochastic game, which in turn determine the optimal strategies of the players.
\end{abstract}

\maketitle

\section{Introduction}
Two-player zero-sum stochastic games between a stopper and a singular-controller formalise situations where two players, a stopper and a singular-controller, observe an underlying process $X$ on a finite-time horizon. The first player (stopper) chooses the stopping time when to end the game and receives a payoff from the second player (controller). The payoff depends on the position of $X$ at the end of the game, on the sample path of $X$ in the interval of time when the game is played, and on a cost proportional to the amount of control exerted by the controller. In the Introduction of \cite{bovo2024halfline} and \cite{hernandez2015zero}, the authors show possible applications of this class of games in the context of financial mathematics: the latter describes a model in which a central bank (controller) optimises over a payoff by choosing the exchange rate up to a possible political veto (stopper). The former considers, e.g., a worst-case scenario model in which a portfolio is optimised by an agent (controller) up to the least favorable stopping time (stopper).

In recent years, this class of stochastic games has been studied extensively by Bovo, De Angelis, Issoglio and Palczewski in different setups \cite{bovo2024halfline,bovo2024saddle,bovo2024variational,bovo2023,bovo2023b}.
In \cite{bovo2024variational}, the authors prove that there exists a value function for the game, which is the maximal solution of a system of variational inequalities of min-max type, and provide an optimal strategy for the stopper. The Markovian structure of the problem allows an implicit characterisation of the region where the first player is better off stopping. 
The contributions \cite{bovo2023,bovo2023b} extend some results from \cite{bovo2024variational} by proving that the game admits a value and by providing an optimal strategy for the stopper in the cases where either the controller is allowed to act in only a few directions of the state space or where the diffusion coefficient of the underlying process is degenerate. 
In \cite{bovo2024halfline}, the attention is focused on the case where the underlying process is one-dimensional and lives in a bounded domain with an absorbing boundary. The paper follows ideas from \cite{bovo2024variational} but requires additional work due to a technical issue arising from the presence of the absorbing boundary. 
In \cite{bovo2024saddle}, the authors find a saddle point of the game on a finite-time horizon and where the underlying process is one-dimensional. They connect the partial derivative with respect to the space variable of the value function with an optimal stopping problem with absorption. This connection is used to characterise the free-boundary $b$, and to provide the optimal strategy for the controller, which is implemented by reflecting the underlying process along $b$. The current paper contributes to this stream of literature by addressing questions about the differentiability of the value function in the whole domain; we are particularly interested in the regularity of the value function across the free-boundaries $a$ (associated to the stopping region of the original game) and $b$, which requires some technical results such as the continuity of the hitting times with respect to the starting points. These boundaries split the domain into three regions: the action region $\cM\coloneqq\cI^c$ (delimited by $b$), the stopping region $\cS\coloneqq\cC^c$ (delimited by $a$) and $\cC\cap\cI$, whose precise definitions are given in Section \ref{sec:main_thm}.

We prove that the value function is of class $C^1$ (jointly in time and space) in the whole domain $[0,T)\times\cO$ (where $\cO$ is the domain of $X$) and we show that the second-order spatial derivative and the second-order mixed derivative are continuous everywhere except for a (potential) jump across the stopping boundary (Theorem \ref{thm:main}). The latter discontinuity is a natural consequence of the partial differential equation (PDE) associated to the problem (see Remark \ref{rmk:v_xx-disc}). The continuity of the spatial derivative is a direct result from \cite{bovo2024halfline,bovo2024variational} whereas the continuity of the time derivative is obtained in here in two steps: we first provide this result at the boundary $a$ (Theorem \ref{thm:v_tC0}), then we prove it at the boundary $b$. The latter is obtained as a by-product of the $C^1$-regularity of $v_x$ (Theorems \ref{thm:v_txCont} and \ref{thm:v_xxCont}). We recall that $v_x$ is the value function of an auxiliary optimal stopping problem with absorption (see \cite{bovo2024saddle}). Therefore, we adopt methods from optimal stopping theory for the analysis of the stochastic game and the auxiliary optimal stopping problem. In particular, we study the regularity of the entry times in $\cS$ and $\cM$ as a function of the initial time and starting point. We want to emphasise that, beyond its analytical value, the regularity we obtain for the value function provides a foundation for further studying these games: it allows the application of (generalised) It\^o's formula that leads to the analysis of the free-boundaries, their integral equations, and numerical approximations of the value function. Our work is a natural continuation of the existing results on this class of games.

Recently, stochastic games of stopper vs.\ controller type have received increasing attention. In 1996, Maitra and Sudderth \cite{maitra1996gambler} provided the first contribution for a class of games in discrete time. Then, the problem was studied in a continuous time framework in which the controller is allowed to use `classical' controls, i.e., functions which are absolutely continuous with respect to the Lebesgue measure of time (see, e.g., \cite{bayraktar2013controller, bayraktar2011regularity, choukroun2015bsde}). In those contributions, the underlying process is one-dimensional and the game is set on an infinite-time horizon. Since the admissible controls are `classical', the one-dimensionality of the problem leads to a free-boundary problem of obstacle type with a Hamiltonian. Thus, there is only one free-boundary in those problems and it is the one associated to the stopper. In particular, it is equal to the union of isolated points of the real line. In \cite{bayraktar2013controller, choukroun2015bsde}, the value functions is proved to be continuous, but there are no results on its differentiability.
In \cite{bayraktar2011regularity}, the authors consider a specific framework where they are able to connect the game with a classical control problem; the latter consists of minimising the probability of lifetime ruin when the consumption rate is random and the optimiser invests in a financial market. The value function of the game is obtained via viscosity approach and it is shown to be globally $C^1$, and globally $C^2$ but for a finite number of points. Finally, Hernandez-Hernandez and co-authors study the game when the control is singular (as function of time), like in our setting, but in the infinite-time horizon framework \cite{hernandez2015zero, hernandez2015zsgsingular}. In the former paper, they provide a verification theorem for the stochastic game and show some examples with an explicit solution, in particular one where the value function is merely continuous; in the latter paper, they consider an underlying L\'evy process and obtain continuous differentiability (almost everywhere) of the value function by imposing some structural conditions.

In all the aforementioned works in continuous time, the time-horizon is not finite, and the underlying process is one-dimensional. Combining these two assumptions, the variationl inequality associated to those problems is composed of an ordinary differential equation, and the associated free-boundaries are union of isolated points. Similarly, we work with a one-dimensional diffusion but on a finite-time horizon. Our value function is thus connected with a partial differential equation of parabolic type so that: (i) the continuous differentiability of the value function requires differentiability in two variables and (ii) the free-boundaries are curves in $\R^2$. Therefore, the main challenge of this work is the regularity of the value function across these moving barriers. 

Some examples of two-player stochastic games between a stopper and a singular-controller have been studied also in the nonzero-sum framework; in particular, in the context of stochastic games with uncertain competition (also called `ghost' games), where one of the two players (the ghost) takes part in the game only with some probability \cite{2022salami,ekstrom2023finetti}. In the former paper a `ghost' controller plays against a stopper, whereas in the latter paper a `ghost' stopper plays against a controller. However, these are only some specific examples and - to the best of our knowledge - nonzero-sum stochastic games between a stopper and a singular-controller have not been studied in a general setup yet.

Our paper is organised as follows. In Section \ref{sec:setting}, we formulate the problem and present the main assumptions of our analysis. In Section \ref{sec:main_thm}, we give detailed references to properties previously obtained on this problem and we state the main contribution of our work. Its proof is distilled in the rest of the paper. In Section \ref{sec:vtcont}, we prove continuity of the time derivative across the boundary $a$. In section \ref{sec:convergence}, we prove the a.s.\ convergence of the optimal stopping times associated to the auxiliary problem as functions of the initial time and space values. In Section \ref{sec:vxdiff}, we obtain the differentiability in time and space of $v_x$ across the boundary $b$ and as by-product the differentiability in time of $v$.

\section{Problem formulation and main assumptions}\label{sec:setting}

In this section, we define our problem and its assumptions. Let $(\Omega,\cF,\F,\P)$ be a filtered probability space and let $(W_t)_{t\geq0}$ be a one-dimensional Brownian motion adapted to $\F=(\cF_t)_{t\geq0}$. Fix a time horizon $T\in(0,\infty)$. The uncontrolled stochastic dynamics $(X^0_t)_{t\in[0,T]}$ is the unique (non-explosive) $\F$-adapted solution of the one-dimensional stochastic differential equation (SDE):
\begin{align}\label{eq:SDE0}
X_t^{0}=x+\int_0^t\mu(X_s^{0}) \ud s+ \int_0^t\sigma(X_s^{0})\ud W_s,\quad t\in[0,T].
\end{align}
The process $X^0$ evolves in an open unbounded domain $\cO$. In particular, we assume that either $\cO=\R$ or $\cO=(0,\infty)$, and we write $\overline{\cO}=\R$ or $\overline{\cO}=[0,\infty)$ with a slight abuse of notation. Assumptions on the functions $\mu:\cO\to\R$ and $\sigma:\cO\to(0,\infty)$ are stated later in Assumption \ref{ass:1} and they will guarantee the existence of a unique solution of \eqref{eq:SDE0}. For future reference, we denote by $\cL$ the infinitesimal generator associated to $X^0$, i.e.,
	\begin{equation}\label{eq:cL}
		(\cL \varphi)(t,x)\coloneqq\mu(x)\varphi_x(t,x)+\tfrac{\sigma^2(x)}{2}\varphi_{xx}(t,x),
	\end{equation}
for any sufficiently regular function $\varphi:[0,T]\times\cO\to\R$.

Admissible controls will be drawn from the class of bounded variation c\`adl\`ag processes. Precisely, for $(t,x)\in[0,T]\times\cO$, the class of admissible controls is defined as
\begin{align*}
\cA_{t,x}\coloneqq \left\{\nu \left|
\begin{aligned}\,&\,\nu_s=\nu_s^+-\nu_s^-\text{ for all $s\geq 0$ where $(\nu_s^+)_{s\geq0}$ and $(\nu_t^-)_{t\geq 0}$ are $\F$-adapted,}\\
\,&\text{real-valued, non-decreasing, c\`adl\`ag, with $\nu_{0-}^\pm=0$ and $\E[|\nu_{T-t}^\pm|^2]<\infty$;}\\
\,&\text{if $\cO=(0,\infty)$, then $X^\nu_{\tau_\cO}=0$, $\P$-a.s.\ on $\{\tau_\cO<T-t\}$;}
\end{aligned} 
\right. \right\}\!,
\end{align*}
where the controlled dynamics is defined as follows: letting $\tau_\cO\coloneqq\inf\{s\ge 0:X^\nu_s\notin\cO\}\wedge(T-t)$, the process $(X_s^\nu)_{s\in[0,T-t]}$ is the unique $\F$-adapted solution of
\begin{align}\label{eq:SDE}
\begin{aligned}
X_s^{\nu}&=x+\int_0^s\mu(X_v^{\nu}) \ud v+ \int_0^s\sigma(X_v^{\nu})\ud W_v+\nu_s,\quad s\in[0,\tau_\cO],\\
\end{aligned}
\end{align}
with $X_s^\nu=0$, for $s\in(\tau_\cO,T-t]$ (with the convention $(a,a]=\varnothing$), meaning that $X^\nu_s(\omega)=0$ for $s\in(\tau_\cO(\omega),T-t]$ for $\P$-a.e.\ $\omega$.

Notice that $X_{0-}=x\in\cO$ and sometimes we will use the notation $X^{\nu;x}$ to indicate the starting point of the process $X^{\nu}$. 
Finally, the class of admissible stopping times is given by
\begin{align*}
\cT_t\coloneqq\{\tau\,|\,\tau\text{ is a $\F$-stopping time with $\tau\in[0,T-t]$, $\P$-a.s.}\}.
\end{align*}

Given $(t,x)\in[0,T]\times \cO$ and an admissible pair $(\tau,\nu)\in\cT_t\times\cA_{t,x}$ we consider an expected payoff of the form 
\begin{align}\label{eq:payoff}
\cJ_{t,x}(\tau,\nu)\coloneqq\E_x\Big[\e^{-r\tau}g(t+\tau)+\int_{0}^\tau \!\e^{-rs}h(t+s,X_s^{\nu})\,\ud s+\int_{[0,\tau]}\!\e^{-rs}\alpha_{0}\,\ud |\nu|_s\Big]
\end{align}
where $r\geq0$ and $\alpha_0>0$ are given constants, $g:[0,T]\to \R$, $h:[0,T]\times\cO\to \R$ are suitable functions (whose properties will be described in Assumption \ref{ass:2}) and $|\nu|_s$ denotes the total variation of $\nu$ on the interval $[0,s]$. The function $\cJ_{t,x}$ is the expected payoff in a zero-sum game in which the {\em maximiser} picks a stopping time $\tau$ and the {\em minimiser} picks an admissible control $\nu$. As usual we introduce the upper and lower value of the game, denoted $\overline v$ and $\underline v$, respectively, and defined as 
\[
\overline v(t,x)=\adjustlimits\inf_{\nu\in\cA_{t,x}}\sup_{\tau\in\cT_t}\cJ_{t,x}(\tau,\nu)\quad\text{and}\quad \underline v(t,x)=\adjustlimits\sup_{\tau\in\cT_t}\inf_{\nu\in\cA_{t,x}}\cJ_{t,x}(\tau,\nu).
\] 
In general, $\underline v\le\overline v$ but Assumptions \ref{ass:1} and \ref{ass:2} guarantee the existence of a saddle point for the game so that $\underline v=\overline v$.

\begin{definition}\label{def:NE}
Fix $(t,x)\in[0,T]\times\cO$. An admissible pair $(\tau_*,\nu^*)\in\cT_t\times\cA_{t,x}$ is a saddle point for the game with payoff \eqref{eq:payoff}, evaluated at $(t,x)$, if
\begin{align*}
\cJ_{t,x}(\tau,\nu^*)\le\cJ_{t,x}(\tau_*,\nu^*)\le \cJ_{t,x}(\tau_*,\nu),\quad\text{for all $(\tau,\nu)\in\cT_t\times\cA_{t,x}$}.
\end{align*}
\end{definition} 

Before giving the main assumptions of the paper, we introduce some notation about functional spaces that are used in the paper. For $\gamma\in(0,1)$, we denote by $C^{\gamma}(D)$ the set of $\gamma$-H\"older continuous functions on $D\subset[0,T]\times\cO$ with respect to the parabolic distance. We denote by $C^{0,1;\gamma}(D)$ the set of functions in $C^{\gamma}(D)$ whose derivative with respect to space is also an element of $C^{\gamma}(D)$ and we use $C^{1,2;\gamma}(D)$ to indicate functions whose first derivatives with respect to time and space, and second derivative with respect to space belong to $C^{\gamma}(D)$. We also recall the set of functions $W^{1,2;p}(D)$ whose first order derivatives and second order spatial derivative, possibly in a weak sense, belong to $L^p(D)$, for some $p\in[1,\infty]$. The compact embedding $W^{1,2;p}(D)\hookrightarrow C^{0,1;\gamma}(D)$, with $p>3$, $\gamma={3/p}$ and $D\subset\R$ (see, e.g., \cite[Eq.\
(E.9)]{fleming2012deterministic}) implies that functions in the space $W^{1,2;p}(D)$ admits a continuous version of its spatial derivative. Finally, we use the subscript $\ell oc$ for example in $C^{0,1;\gamma}_{\ell oc}([0,T]\times \cO)$ to indicate the set of functions belonging to $C^{0,1;\gamma}(\cK)$ for any compact $\cK\subset[0,T]\times \cO$. A detailed presentation of these sets can be found in \cite[Sec.\ 2]{bovo2024variational}.

The next set of assumptions are needed to recall results from \cite{bovo2024halfline,bovo2024saddle,bovo2024variational}.

\begin{assumption}[{\bf Dynamics}]\label{ass:1}
The following conditions hold:
\begin{itemize}
\item[i)] The function $\mu$ has linear growth, it is twice continuously differentiable and convex. Its derivative $\mu_x$ is bounded from above and there exist $C>0$ and $p\in[1,\infty)$ such that, for every $x,y\in\cO$, we have
\begin{equation}\label{eq:mu_x-lip}
 |\mu_x(x)-\mu_x(y)|\leq C(1+|x|^p+|y|^p)|x-y|.
\end{equation}
Moreover, $\mu_{xx}\in C^\gamma_{\ell oc}(\cO)$ for some $\gamma\in(0,1)$;

\item[ii)] If $\cO=(0,\infty)$, then $\mu(0)=0$ and $\mu_x(0)=\lim_{x\downarrow 0}\mu_x(x)$ finite (hence, $\mu$ is Lipschitz);

\item[iii)] The function $\sigma$ is such that $\sigma(x)=\sigma_0$ when $\cO=\R$, and $\sigma(x)=\sigma_1 x$ when $\cO=(0,\infty)$, for some $\sigma_0,\sigma_1\in(0,\infty)$.
\end{itemize} 
\end{assumption}

\begin{remark}\label{rmk:2gBm}
Notice that, when $\cO=(0,\infty)$, Assumption \ref{ass:1} implies that $\mu_x(0)x\leq \mu(x) \leq \bar\mu_x x,$ for every $x\in\cO$, where $\bar \mu_x\coloneqq\sup_{x\in\cO}\mu_x(x)$. Then, there exist two geometric Brownian motions
$$\bunderline{X}_t\coloneqq x\exp\big\{(\mu_x(0)-\sigma_1^2/2)t+\sigma_1 W_t \big\} \quad \text{and} \quad \boverline X_{\!\!t}\coloneqq x\exp\big\{(\bar\mu_x-\sigma_1^2/2)t+\sigma_1 W_t \big\},$$
such that, by standard comparison principles for SDEs, we obtain $0<\bunderline{X}_t\leq X^0_t\leq \boverline X_{\!\!t}<\infty,$ for every $t\geq 0$, $\P$-a.s. That is, $X^0_t\in\cO$ for every $t\geq 0$, $\P$-a.s.
\end{remark}

Letting 
\begin{equation}\label{eq:Theta}
	\Theta(t,x)\coloneqq \dot g(t)-rg(t)+h(t,x),\quad \text{for $(t,x)\in[0,T]\times\overline\cO$},
\end{equation}
we set the following further assumptions, also required to recall results from \cite{bovo2024halfline,bovo2024saddle,bovo2024variational}.
\begin{assumption}[{\bf Payoff}]\label{ass:2}
For some $\gamma\in(0,1)$ and $c>0$, the following conditions hold:
\begin{itemize}
 \item[ i)] The function $g$ is continuously differentiable with $\dot g\in C^\gamma([0,T])$;
	\item[ ii)] The function $h\in C^\gamma_{\ell oc}([0,T]\times\cO)$ with $|h(t,x)|\le c(1+|x|^2)$; the spatial derivative $h_x\in C^{0,1;\gamma}_{\ell oc}([0,T]\times\cO)$ and it is non-negative, with $x\mapsto h_x(t,x)$ non-decreasing, $t\mapsto h_x(t,x)$ non-increasing and $|h_x(t,x)|\le c(1+|x|^2)$; moreover, for every $t\in[0,T)$ there exists $L(t)$ such that for every $x,y\in\cO$ we have
 \begin{equation}\label{eq:h_x-lip}
 |h_x(t,x)-h_x(t,y)|\leq L(t)|x-y|,
 \end{equation}
where the function $t\mapsto L(t)$ is locally bounded;
 \item[iii)] The function $t\mapsto \Theta(t,x)$ is non-increasing for all $x\in\overline\cO$ and, when $\cO=(0,\infty)$, $\Theta(0,0)< 0$;
\item[iv)] There is $K_1>0$ such that 
$
\inf_{(t,x)\in[0,T]\times\cO}\Theta(t,x)= -K_1;
$
\item[v)] For all $t\in[0,T)$ we have $\sup_{x\in\cO}\Theta(t,x)>0$;
\item[vi)] There exist $C>0$ and $p\in[1,\infty)$ such that for $t>s$ and $x\in\cO$ we have
\begin{align}\label{eq:Theta_Lip_time}
\Theta(t,x)-\Theta(s,x)\geq -C(t-s)(1+|x|^p),
\end{align}
and
\begin{align}\label{eq:hx_Lip_time}
h_x(t,x)-h_x(s,x)\geq -C(t-s)(1+|x|^p).
\end{align}
\end{itemize}
\end{assumption}

\section{Some previous results and the main theorem}\label{sec:main_thm}
In this section we recall some results from \cite{bovo2024halfline,bovo2024saddle,bovo2024variational} that are essential for the presentation of our main contribution. We will make precise references whenever we recall results from those contributions to help the interested reader. We conclude the section providing the main theorem of the paper.

\subsection{Initial properties of the value function}\label{sec:InitProp}
Assumptions \ref{ass:1} and \ref{ass:2} guarantee that results from \cite{bovo2024halfline}, \cite{bovo2024saddle} and \cite{bovo2024variational} hold. In particular, we have from \cite[Th.\ 2.3]{bovo2024halfline} (for the case $\cO=(0,\infty)$) and from \cite[Th.\ 2.3]{bovo2024variational} (for the case $\cO=\R$) that the value of the game $v$ exists, 
\begin{align}\label{eq:reg_v_vx}
v\in C([0,T]\times\overline{\cO})\cap W^{1,2;p}_{\ell oc}([0,T)\times\cO)\text{, }\;\text{and }\;v_x\in C^{\gamma}_{\ell oc}([0,T)\times\cO)
\end{align}
thanks to the compact embedding of $W^{1,2;p}([0,T)\times\cO)\hookrightarrow C^{0,1;\gamma}([0,T)\times\cO)$. Moreover, $v$ solves a.e.\ the system of variational inequalities
\begin{align*}
\begin{cases}
\max\{\min\{v_t+\cL v-rv +h, \alpha_0-|v_x|\}, g-v\}=0\\
\min\{\max\{v_t+\cL v-rv +h, g-v\}, \alpha_0-|v_x|\}=0,
\end{cases}
\end{align*}
where $\cL$ was defined in \eqref{eq:cL} and with $v(T,x)=g(T)$ for all $x\in\cO$ and $v(t,0)=g(t)$ for all $t\in[0,T)$ (whenever $\cO=(0,\infty)$). In particular, it satisfies two constraints: $v(t,x)\ge g(t)$ and $v_x(t,x)\leq \alpha_0$ for all $(t,x)\in[0,T]\times\cO$. We can characterise two regions of the space, the continuation region $\cC$ and the inaction region $\cI$, defined as 
\begin{align*}
\cC\coloneqq \{(t,x)\in[0,T]\times\cO\,|\, v(t,x)>g(t)\}\quad\text{and}\quad\cI\coloneqq \{(t,x)\in[0,T]\times\cO\,|\,v_x(t,x)<\alpha_0\},
\end{align*}
respectively. The complementary sets are denoted by $\cS\coloneqq\cC^c$ and $\cM\coloneqq\cI^c$, the stopping set and the action set, respectively. The motivation for those names will be clear in \eqref{eq:saddle_def} when we will present a saddle point of the game.
 
We now recall a result about the monotonicity of the functions $v$ and $v_x$ with respect to the variables $t$ and $x$.
\begin{proposition}\label{prop:monotonicity} The following results hold:
\begin{itemize}
\item[i)] The mapping $x\mapsto v(t,x)$ is non-decreasing and convex, for every $t\in[0,T]$;
\item[ii)] The mapping $t\mapsto v(t,y)-g(t)$ is non-increasing for every $y\in\cO$;
\item[iii)] The mapping $t \mapsto v_x(t,y)$ is non-increasing for every $y\in\cO$.
\end{itemize}
\end{proposition}
\begin{proof}
The proof of monotonicity and convexity of $x\mapsto v(t,x)$, for every $t\in[0,T]$, can be found in \cite[Prop.\ 3.1]{bovo2024saddle} and in \cite[Prop.\ 3.5]{bovo2024saddle}, respectively; the proof of monotonicity of $t\mapsto v(t,y)-g(t)$ and $t \mapsto v_x(t,y)$, for all $y\in\cO$, can be found in \cite[Prop.\ 3.7]{bovo2024saddle} and in \cite[Prop.\ 4.5]{bovo2024saddle}, respectively.
\end{proof}

\subsection{Free-boundaries} Proposition \ref{prop:monotonicity} allows a new parametrisation of the sets $\cC$ and $\cI$. For $t\in[0,T)$, we set
\begin{align*}
a(t)\coloneqq \inf\{y\in\cO:v(t,y)>g(t)\}\quad\text{and}\quad b(t)\coloneqq\sup\{y\in\cO:v_x(t,y)<\alpha_0\},
\end{align*}
with the convention that $\inf\varnothing=\sup\cO=+\infty$ and $\sup\varnothing=\inf\cO$. Then, by continuity of $v$ and $v_x$ \eqref{eq:reg_v_vx}, we obtain
\begin{align}\label{eq:C-I-boundaries}
\qquad \cC= \{(t,x)\in[0,T)\times\cO\,|\, x>a(t)\}\quad\text{and}\quad\cI= \{(t,x)\in[0,T)\times\cO\,|\, x<b(t)\}\cup(\{T\}\times\cO).
\end{align}

This characterisation allows a clearer representation of $v$ as a classical solution of the free-boundary problem
\begin{align}\label{eq:PDE}
\begin{aligned}
\begin{cases}v_t(t,x)+\cL v(t,x)-rv(t,x) = -h(t,x),&(t,x)\in\cC\cap\cI,\\
v(t,a(t))=g(t)\quad\text{and}\quad v_x(t,a(t))=0,& t\in[0,T),\\
v_x(t,b(t))=\alpha_0,& t\in[0,T),\\
v(T,x)=g(T),& x\in\cO.
\end{cases}
\end{aligned}
\end{align}

The regularity of $a$ and $b$ is studied in \cite[Sect.\ 6]{bovo2024saddle}. To recall those results (Proposition \ref{prop:monotone}), we introduce the (uncontrolled) diffusion $Y^y=(Y^y_t)_{t\in[0,T]}$, defined by the stochastic differential equation
\begin{align}\label{eq:SDE_Y}
Y_t=y+\int_0^t \big(\mu(Y_s)+\sigma(Y_s)\sigma_x(Y_s)\big)\ud s+ \int_0^t\sigma(Y_s) \ud W_s,\quad t\in[0,T],
\end{align}
and the corresponding exit time from $\cC$ as 
\begin{align}\label{eq:taua}
\tau_a=\tau_a(t,y)\coloneqq \inf\{s\ge 0|Y_{s}^y\le a(t+s)\}\wedge(T\!-\!t).
\end{align}
As it will be clear from the next section, the process $Y^y$, whose definition is anticipated here only to introduce Assumption \ref{ass:3} and the consequent Proposition \ref{prop:monotone}, will serve as the underlying diffusion of an auxiliary optimal stopping problem. Using $\Theta$ from \eqref{eq:Theta} and letting $\underline{\Theta}(t)\coloneqq\inf\{ y\in\cO:\Theta(t,y)>0\}$, under the following additional Assumption \ref{ass:3}, the regularity results from Proposition \ref{prop:monotone} hold.
\begin{assumption}\label{ass:3}
The following conditions hold:
\begin{itemize}
 \item[i)] The function $h_x(t,x)>0$ for any $x\le \underline{\Theta}(t)$ and all $t\in[0,T]$;
 \item[ii)] For all $t\in[0,T)$
\begin{align*}
\liminf_{y\to\infty}\E\Big[\int_0^{\tau_{a}(t,y)}\e^{-\int_0^s(r-\mu_x(Y_u^y))\ud u}h_x(t+s,Y_s^y)\ud s\Big]>\alpha_0;
\end{align*}
 \item[iii)] It holds $h_{xx}(t,b(t))+\mu_{xx}(b(t))>0$ for all $t\in[0,T)$.
\end{itemize}
\end{assumption}

\begin{proposition}\label{prop:monotone}
The following results hold:
\begin{itemize}
 \item[i)] The mapping $t\mapsto a(t)$ is non-decreasing, continuous on $[0,T]$ with $a(T-)=\underline{\Theta}(T)$, and $a(t)\le \underline{\Theta}(t)$ for all $t<T$;
 \item[ii)] The mapping $t\mapsto b(t)$ is non-decreasing and continuous on $[0,T]$ with $b(t)<\infty$ for all $t\in[0,T)$.
\end{itemize}
\end{proposition}

\begin{proof}
The proof of monotonicity and continuity of $a$ on $[0,T]$ with limit $a(T-)=\underline{\Theta}(T)$ is obtained combining results from \cite[Prop.\ 3.7]{bovo2024saddle} and \cite[Cor.\ 6.3]{bovo2024saddle}. The proof of monotonicity and continuity of $b$ on $[0,T]$ with $b(t)<\infty$ for all $t\in[0,T]$ is obtained combining \cite[Cor.\ 4.6]{bovo2024saddle} and \cite[Cor.\ 6.7]{bovo2024saddle}.
\end{proof}
By continuity of $v_x$, it is not difficult to notice that $a(t)<b(t)$ for any $t\in[0,T)$ such that $b(t)>\inf\cO$ and $a(t)<\sup\cO$. Proposition \ref{prop:monotone} guarantees that $a(t)\leq \underline{\Theta}(t)<\sup\cO$ for every $t\in [0,T)$. We assume that $b(t)>\inf\cO$ for every $t\in[0,T)$. Some sufficient conditions that guarantee this result can be found in \cite[Prop.\ 4.8]{bovo2024saddle}. Therefore,
\begin{equation}\label{eq:a<b}
 a(t)<b(t), \qquad t\in[0,T).
\end{equation}

\subsection{An auxiliary optimal stopping problem}
In \cite[Sec.\ 4]{bovo2024saddle}, it is shown that $v_x$ admits a probabilistic representation, in the form of an optimal stopping problem, which we describe in this section. Heuristically, this is obtained by differentiating the PDE in \eqref{eq:PDE} with respect to the variable $x$ and connecting the newly PDE to an optimal stopping problem with absorption. This approach mimics the renowned connection between singular control and optimal stopping problems, introduced in \cite{karatzas1984bridge1, karatzas1985bridge2}, but its formal derivation in the context of stochastic games requires some additional formal proofs.

Under our assumptions, the SDE \eqref{eq:SDE_Y} admits a unique strong solution, denoted by $Y=Y^y$. The end-points of $\cO$ remain unattainable for the process $Y$: (i) for $\cO=\R$, the SDE is non-explosive, (ii) for $\cO=(0,\infty)$, we have that $\sigma(y)\sigma_x(y)= \sigma_1^2 y$ is linear and so a similar observation as in Remark \ref{rmk:2gBm} still holds. Furthermore, we introduce an additional mild assumption, which will only be needed in Proposition \ref{prop:v_x-Lip-time} and Theorem \ref{thm:v_txCont}.
\begin{assumption}\label{ass:Ypdf}
 For every $(t,y)\in (0,T]\times\cO$, the probability density function $p(\,\cdot\,;t,y)$ of the random variable $Y^y_t$ is continuously differentiable on $\cO$.
\end{assumption}

\begin{remark}
 Notice that when $\cO=\R$, since $\sigma(x)=\sigma_0\in(0,\infty)$ for every $x\in\cO$ in our setting, the result of Assumption \ref{ass:Ypdf} is already guaranteed by, e.g., \cite[Th.\ 1.2]{menozzi2021density}. Therefore, in fact, we only need Assumption \ref{ass:Ypdf} when $\cO=(0,\infty)$.
\end{remark}

We denote by $\cG$ the infinitesimal generator associated to $Y$, i.e., 
\begin{equation}\label{eq:cG}
	(\cG \varphi)(t,x)\coloneqq (\mu(x)+\sigma(x)\dot{\sigma}(x))\varphi_x(t,x)+\tfrac{\sigma^2(x)}{2}\varphi_{xx}(t,x),
\end{equation}
for any sufficiently regular function $\varphi:[0,T]\times\cO\to\R$, and we define 
\begin{align}\label{eq:lambda}
\lambda(y)\coloneqq r-\mu_x(y)\quad\text{for $y\in\cO$},
\end{align}
which will serve as the discount factor in \eqref{eq:optstop} below.

We know that $v_x$ is a bounded continuous function, with $0\leq v_x\leq \alpha_{0}$. Moreover, as shown in \cite[Th.\ 4.2]{bovo2024saddle}, it solves
\begin{align}\label{eq:ineqPDE_vx}
\big(v_{tx}+\cG v_x-\lambda v_x\big)(t,x)\geq -h_x(t,x),\quad (t,x)\in\cC,
\end{align}
in the sense of distributions and 
\begin{align}\label{eq:PDE_vx}
\big(v_{tx}+\cG v_x-\lambda v_x\big)(t,x)= -h_x(t,x),\quad (t,x)\in\cC\cap\cI,
\end{align}
in the classical sense.

The PDE \eqref{eq:PDE_vx} suggests the probabilistic representation for $v_x$, in the form of an optimal stopping problem. As proved in \cite[Th.\ 4.3]{bovo2024saddle}, for every $(t,y)\in[0,T]\times \cO$, we have that
\begin{align}\label{eq:optstop}
v_x(t,y)=\inf_{\tau\in\cT_t}\E\Big[\e^{-\int_0^\tau\lambda(Y_u^y)\ud u}\alpha_{0}\mathds{1}_{\{\tau< \tau_a(t,y)\}}+\int_{0}^{\tau\wedge\tau_a(t,y)}\!\e^{-\int_0^s\lambda(Y_u^y)\ud u}h_x(t+s,Y_s^y)\,\ud s\Big],
\end{align}
where $\tau_a(t,y)$ was defined in \eqref{eq:taua}, and that the stopping time $\sigma_*=\sigma_*(t,y)$ defined by
\begin{align}\label{eq:sigma*}
\sigma_*\coloneqq\inf\{s\ge 0 :v_x(t+s,Y_s^{y})\geq\alpha_{0}\}\wedge(T-t),
\end{align}
is optimal. Clearly, $\sigma_*$ admits the representation $\sigma_*\coloneqq\inf\{s\ge 0 :Y_s^{y}\ge b(t+s)\}\wedge(T-t)$.

\subsection{Optimal strategies}
We are now ready to present the two optimal strategies, one for the stopper and one for the controller. For the former, we have 
\begin{align}\label{eq:tau*}
\tau_*\coloneqq \inf\big\{s\ge 0:\min\big[v(t+s,X_s^{\nu;x}), v(t+s,X_{s-}^{\nu;x})\big]\le g(t+s)\big\}\wedge (T-t).
\end{align}
by \cite[Th.\ 2.3]{bovo2024halfline} (for the case $\cO=(0,\infty)$) and by \cite[Th.\ 2.3]{bovo2024variational} (for the case $\cO=\R$). 

The optimal strategy for the controller is described as follows. In \cite[Sec.\ 5]{bovo2024saddle} it is proved that there exists a pair $(X^{\nu^*},\nu^*)$ solution of the reflecting SDE along the boundary $b$ and that $\nu^*$ is an optimal control for the game. That is, given an initial point $(t,x)\in[0,T)\times\cO$, the pair $(X^{\nu^*},\nu^*)$ is such that:
\begin{itemize}
\item[(i)] $(X^{\nu^*},\nu^*)$ is $\F$-adapted;
\item[(ii)] $X^{\nu^*\!;x}_{0}=b(t)\wedge x$, $s\mapsto X^{\nu^*\!;x}_s\in\cO$ is continuous and its dynamics is given by \eqref{eq:SDE};
\item[(iii)] $\nu^*_0=-(x-b(t))^+$ and $s\mapsto\nu^*_s$ is continuous and non-increasing for $s\in[0,T-t]$;
\item[(iv)] The following conditions hold $\P$-a.s.:
\begin{align}\label{eq:refcond}
\begin{aligned}
X^{\nu^*\!;x}_s\le b(t+s),\quad\text{for $s\in[0,T-t]$} \quad \text{and} \quad
\int_{(0,T-t]}\mathds{1}_{\big\{X^{\nu^*\!;x}_s<b(t+s)\big\}}\ud \nu^*_s=0.
\end{aligned}
\end{align} 
\end{itemize} 
Moreover, $(X^{\nu^*},\nu^*)$ is unique, up to indistinguishability and can be expressed by
\begin{align}\label{eq:RSDE}
\begin{aligned}
&X^{\nu^*\!;x}_s=x+\int_0^s \mu(X^{\nu^*\!;x}_u)\ud u+\int_0^s\sigma(X^{\nu^*\!;x}_u)\ud W_u+\nu^*_s,\\
&\nu^*_s=-\sup_{0\le u\le s}\Big(x+\int_0^u \mu(X^{\nu^*\!;x}_v)\ud v+\int_0^u\sigma(X^{\nu^*\!;x}_v)\ud W_v-b(t+u)\Big)^+,
\end{aligned}
\end{align}
for $s\in[0,T-t]$, $\P$-a.s., and it holds (see \cite[Lem.\ 5.2]{bovo2024saddle})
\begin{align}\label{eq:nuL2}
\E\Big[\sup_{0\le s\le T-t}\big|X^{\nu^*\!;x}_s\big|^2+\big|\nu^*_{T-t}\big|^2\Big]<\infty.
\end{align}
Repeating the steps of the proof of \cite[Lem.\ 5.2]{bovo2024saddle} it is possible to extend the $L^2$-estimate therein to a $L^p$-estimate for any $p\in[1,\infty)$. The result is given below without proving it because it follows by standard arguments and repeating almost verbatim the proof given in \cite[App.\ A.2]{bovo2024saddle}.
\begin{corollary}\label{cor:Lpestnu}
Let $(t,x)\in[0,T)\times\cO$ and $(X^{x;\nu^*},\nu^*)$ be the solution of the reflected SDE along $b$. For any $p\in[1,\infty)$, there exists $K_p\in(1,\infty)$ (depending on $T$ and the linear growth of $\mu$ and $\sigma$), such that
\begin{align}\label{eq:nuLp}
\E\Big[\sup_{0\le s\le T-t}\big|X^{x;\nu^*}_s\big|^p+\big|\nu^*_{T-t}\big|^p\Big]<K_p(1+|x|^p).
\end{align}
\end{corollary}

The control $\nu^*$ is optimal and the pair $(\tau_*,\nu^*)\in\cT_t\times\cA_{t,x}$ yields a saddle point, i.e., 
\begin{align}\label{eq:saddle_def}
	\cJ_{t,x}(\tau,\nu^*)\leq \cJ_{t,x}(\tau_*,\nu^*)\leq \cJ_{t,x}(\tau_*,\nu)\qquad\text{for all $(\tau,\nu)\in\cT_t\times\cA_{t,x}$.}
\end{align}
It is worth noticing that the structure of $\nu^*$ implies that the process $X^{\nu^*}$ does not jump upwards. Therefore, at equilibrium, the expression for $\tau_*$ simplifies to 
\[
\tau_*(\nu^*)=\inf\{s\ge 0:X^{\nu^*}_s\le a(t+s)\}\wedge(T-t).
\]
\begin{remark}\label{rem:X-Markov}
 The controlled process $X^{\nu^*;x}$ is a strong Markov process because it is right-continuous and it is a Feller process (see, e.g., \cite[Th.\ 6.1]{baldi2017stochastic}). We recall that $X$ is Feller if
$(t,x)\mapsto\E_{t,x}[f(X_{t+h})]$
 is continuous for all $h>0$ and any $f$ continuous and bounded. The latter holds for the process $X^{\nu^*;x}$ because it is continuous by Kolmogorov's continuity theorem and the estimates that can be obtained by following, e.g., \cite[Ex.\ 1.2.1]{pilipenko2014reflection}.
\end{remark}
\subsection{Main result}

We are now ready to present the main theorem of the paper. Part of its proof will be developed in the next sections.

\begin{theorem}\label{thm:main}
 We have that $v\in C\big([0,T]\times\cO\big)\cap C^1\big([0,T)\times\cO\big)$ and $v_x\in C^1\big(\cC\cup\inter(\cS)\big)$.
\end{theorem}
\begin{proof}
 It was already recalled in Section \ref{sec:InitProp} that $v\in C\big([0,T]\times\cO \big)$ and that $v_x\in C\big([0,T)\times\cO \big)$ (see \eqref{eq:reg_v_vx}). The proof of $v_t\in C\big([0,T)\times\cO\big)$ is obtained by combining the fact that $v_t = \dot g$ on $\inter(\cS)$ together with Theorem \ref{thm:v_tC0} and the fact that $v_t$ is continuous on $\cC\cap\cI$ by \eqref{eq:PDE} together with Corollary \ref{cor:v_t-cont-M}. The proof of $v_x\in C^1\big(\cC\cup\inter{(\cS)}\big)$ is obtained combining Theorem \ref{thm:v_txCont}, Theorem \ref{thm:v_xxCont}, the fact that $v_{tx}$ and $v_{xx}$ are null on $\inter(\cM)\cup \inter(\cS)$, and continuous on $\cC\cap\cI$, since \eqref{eq:PDE_vx} holds, by classical interior regularity results for parabolic PDEs (see, e.g., \cite[Th.\ 3.10]{friedman2008partial}). 
\end{proof}

\section{Continuity of \texorpdfstring{$v_t$}{vₜ} at the boundary \texorpdfstring{$a$}{a}}\label{sec:vtcont}

In this section, we study the regularity of the time derivative $v_t$ at the boundary $a$, its proof is based on the properties of the optimal stopping time $\tau_*$ introduced in \eqref{eq:tau*}. We first prove two preliminary lemmas about $\tau_*$ and then we prove the continuity of $v_t$. Recalling that the stopping region $\cS$ is a closed set, let us introduce the stopping time
\begin{align*}
\tilde{\tau}_*=\tilde{\tau}_*(\nu)=\tilde{\tau}_*(\nu;t,x)\coloneqq&\inf\{s \ge 0:(t+s,X_s^{\nu;x})\in\inter(\cS)\}\wedge (T-t)\\
=&\inf\{s \ge 0:X_s^{\nu;x}<a(t+s)\}\wedge (T-t).
\end{align*}
The next lemma shows that $\tau_*(\nu^*)$ and $\tilde{\tau}_*(\nu^*)$ are equal $\P$-a.s. In other words, this means that the process $X^{\nu^*;x}$ enters immediately the region $\cS$ whenever it touches $\partial\cC$. 
\begin{lemma}\label{lem:tau°}
For all $(t,x)\in[0,T]\times\cO$ and $\nu^*$ as in \eqref{eq:RSDE}, it holds $\tau_*(\nu^*;t,x)=\tilde{\tau}_*(\nu^*;t,x)$, $\P$-a.s. Moreover, let $t<T$ such that $a(t)>\inf\cO$, then $\tau_*(\nu^*;t,a(t))=0$, $\P$-a.s.
\end{lemma}

\begin{proof}
Fix $(t,x)\in[0,T]\times \cO$ and let $\nu^*$ be as in \eqref{eq:RSDE}. The result trivially holds for $t=T$ because $\tau_*(\nu^*;T,x)=\tilde{\tau}_*(\nu^*;T,x)=0$ by definition. It remains to prove the result for $t<T$. Notice that $\tau_*\le \tilde{\tau}_*$ because $\inter(\mathcal{S})\subset \mathcal{S}$. Since the set $[0,T]\times \cO$ can be split in $\cS$ and $\cC$, we divide the proof in three cases depending on whether $(t,x)$ belongs to $\inter(\cS)$, $\partial\cC$ or $\cC$. For almost every $\omega\in\Omega$ the mapping $t\to X_t^{\nu^*;x}(\omega)$ is continuous on $[0,T]$ (see (ii) above \eqref{eq:RSDE}) with a possible downward jump at $0$.

\textbf{Case 1:} $(t,x)\in \inter(\cS)$. It follows that $(t,X_0^{\nu^*;x}(\omega))\in\inter(\cS)$ because $X_{0}^{\nu^*;x}(\omega)\le X_{0-}^{\nu^*;x}(\omega)=x$, so that $\tau_*(\omega)=\tilde{\tau}_*(\omega)=0$. 

\textbf{Case 2:} $(t,x)\in\partial\cC$, i.e., $x=a(t)$ (with $a(t)>\inf\cO$). Recalling $\tau_*\leq \tilde{\tau}_*$, $\P$-a.s., it is sufficient to show that $\tilde{\tau}_*(\omega)=0$. For the clarity of the exposition, we restore in this case the dependence of the stopping times $\tau_*(\nu^*)$ and $\tilde{\tau}_*(\nu^*)$ on the control $\nu^*$. Let $X^{0;x}$ be the solution of the uncontrolled SDE \eqref{eq:SDE}, it holds $X^{\nu^*;x}_s(\omega)\leq X^{0;x}_s(\omega)$ for all $s\in[0,T-t]$ and almost every $\omega$ because, by construction, $\nu^*$ only pushes the process downwards. Therefore, we have $\tilde{\tau}_*(\nu^*)\leq \tilde{\tau}_*(0)$, $\P$-a.s.\ and it is enough to show $\tilde{\tau}_*(0)=0$, $\P$-a.s. By monotonicity of $a$ and \cite[Lem.\ V.46.1]{rogers2000diffusions}, we have that
$$\P\big(\exists \: \eps>0|X^{0;x}_s\geq a(t+s)\text{ for all }s\in[0,\eps]\big)\le\P\big(\exists \: \eps>0|X^{0;x}_s\geq a(t)\text{ for all }s\in[0,\eps]\big)=0,$$ i.e., $\tilde{\tau}_*(0)(\omega)=0$ for almost every $\omega$ and the result holds.

\textbf{Case 3:} $(t,x)\in\cC$. First, if $\tau_*(\omega)=T-t$, then the proof is completed because $T\!-\!t\!=\!\tau_*(\omega)\!\le\! \tilde{\tau}_*(\omega)\!\le\! T\!-\!t.$
Thus, we consider the case $\tau_*(\omega)<T-t$. By \eqref{eq:a<b}, the boundaries are separated, and a potential initial downward jump due to $\nu^*$ keeps the process inside $\cC$ (in particular, $X_0^{\nu^*;x}=b(t)$). After that, the process has continuous paths and so we have $(t+\tau_*(\omega),X_{\tau_*(\omega)}^{\nu^*;x}(\omega))\in\cS$, i.e., $X_{\tau_*(\omega)}^{\nu^*;x}(\omega)=a(\tau_*(\omega))$. By the strong Markov property of $X^{\nu^*;x}$ (Remark \ref{rem:X-Markov}) and \textbf{Case 2}, we have
\begin{align*}
&\P\big(\tau_*(\nu^*;t,x)<\tilde{\tau}_*(\nu^*;t,x)\big|X_{\tau_*}^{\nu^*;x}=a(t+\tau_*)\big)\\
&=\P\big(\tau_*(\nu^*;t+\tau_*,a(t+\tau_*))<\tilde{\tau}_*(\nu^*;t+\tau_*,a(t+\tau_*))\big|X_{0}^{\nu^*;a(t+\tau_*)}=a(t+\tau_*)\big)=0.
\end{align*}
Therefore, we have $\tau_*(\nu^*;t,x)=\tilde{\tau}_*(\nu^*;t,x)$, $\P$-a.s. for any $(t,x)\in[0,T]\times\cO$.
\end{proof}
The next result concerns the continuity of the stopping time $\tau_*(\nu^*)$ with respect to the initial data $(t,x)$ with $\nu^*$ as in \eqref{eq:RSDE}.
\begin{proposition}\label{prop:tauconv}
Let $(t,x)\in[0,T]\times\cO$ and $\nu^*$ as in \eqref{eq:RSDE}. For any $((t_n,x_n))_{n\in\N}\subset [0,T)\times\cO$ such that $(t_n,x_n)\to(t,x)$, it holds $\tau_*(\nu^*;t_n,x_n)\to\tau_*(\nu^*;t,x)$, $\P$-a.s. In particular, if $(t,x)\in\partial \cC$, then $\tau_*(\nu^*;t_n,x_n)\to0$.
\end{proposition}

\begin{proof}
Let $(t,x)\in[0,T]\times\cO$ and $\nu^*$ be the process defined in \eqref{eq:RSDE} associated to the starting point $(t,x)$. If $t=T$, we have $\tau_*(\nu^*;T,x)=0$ and $\tau_*(\nu^*;t_n,x_n)\le T-t_n$ for any $n\in\N$, therefore we have $\tau_*(\nu^*;t_n,x_n)\downarrow \tau_*(\nu^*;T,x)$ as $n\to\infty$.

We consider now the case $t<T$. Fix $\omega\in\Omega$ such that $\tilde{\tau}_*(\omega)=\tau_*(\omega)$ by Lemma \ref{lem:tau°}. For any $s>\tilde{\tau}_*(\omega)$, there exists $\eps\in(\tilde{\tau}_*(\omega),s)$ such that $X_{\eps}^{\nu^*;x}(\omega)< a(t)-\delta$ for some $\delta=\delta(\omega)>0$. For all sufficiently large $n$, we have $X_{\eps}^{\nu^*;x_n}(\omega)<X_{\eps}^{\nu^*;x}(\omega)+\tfrac{\delta}{2}$ and $a(t_n)>a(t)-\frac{\delta}{2}$ by the continuity of the mapping $y\mapsto X^{\nu^*;y}(\omega)$ and $u\mapsto a(u)$, respectively. Combining the two inequalities we obtain
\begin{align*}
X_{\eps}^{\nu^*;x_n}(\omega)<X_{\eps}^{\nu^*;x}(\omega)+\tfrac{\delta}{2}<a(t)-\tfrac{\delta}{2}<a(t_n).
\end{align*}
Thus, $\tau_*(\nu^*;t_n,x_n)(\omega)\leq \eps$ for all sufficiently large $n$ yields $\limsup_{n\to\infty}\tau_*(\nu^*;t_n,x_n)(\omega)<s$. By the arbitrariness of $s$ and $\omega$, we have $\limsup_{n\to\infty}\tau_*(\nu^*;t_n,x_n)\leq \tau_*(\nu^*;t,x)$, $\P$-a.s.

Now, we prove $\liminf_{n\to\infty}\tau_*(\nu^*;t_n,x_n)\geq \tau_*(\nu^*;t,x)$, $\P$-a.s. If $t=T$ or $x\le a(t)$, then $\tau_*(\nu^*;t,x)=0$, $\P$-a.s and the result holds. Let $t<T$ and $x>a(t)$ so that $\tau_*(\nu^*;t,x)>0$, $\P$-a.s. Fix $\omega$ in a set of full measure so that $\tau_*(\nu^*;t,x)(\omega)>0$. By continuity of the maps $u\mapsto X^{\nu^*;x}_u(\omega)$ and $u\mapsto a(t+u)$, for any $s\in(0,\tau_*(\nu^*;t,x)(\omega))$ we have $X_u^{\nu^*;x}(\omega)>a(t+u)$ for all $u\in[0,s]$. In particular,
\begin{align*}
\inf_{u\in[0,s]}\big(X_u^{\nu^*;x}(\omega)-a(t+u)\big)\eqqcolon 2\delta>0.
\end{align*}
Then, by the continuity of $y\mapsto X_u^{\nu^*;y}$ and of $t\mapsto a(t+u)$ uniformly in $u\in[0,s]$, we obtain 
\begin{align*}
\inf_{u\in[0,s]}\big(X_u^{\nu^*;x_n}(\omega)-a(t_n+u)\big)\eqqcolon \delta>0,
\end{align*}
for all $n\in\N$ sufficiently large. Therefore, $\tau_*(\nu^*;t_n,x_n)(\omega)\ge s$ for all sufficiently large $n$ and so $\liminf_{n\to\infty}\tau_*(\nu^*;t_n,x_n)(\omega)\ge s$. By the arbitrariness of $s\in(0,\tau_*(\nu^*;t,x)(\omega))$, we obtain
\begin{align*}
\liminf_{n\to\infty}\tau_*(\nu^*;t_n,x_n)(\omega) \ge \tau_*(\nu^*;t,x)(\omega).
\end{align*}
By the arbitrariness of $\omega$, $\liminf_{n\to\infty} \tau_*(\nu^*;t_n,x_n)\ge \tau_*(\nu^*;t,x)$, $\P$-a.s. Combining this with the first part of the proof, we have $\tau_*(\nu^*;t_n,x_n)\to\tau_*(\nu^*;t,x)$ as $n\to\infty$, $\P$-a.s.
\end{proof}

The last result of this section shows the continuity of $v_t$ at the boundary $a$. Classical interior regularity results for parabolic PDEs \cite[Th.\ 3.10]{friedman2008partial} provide continuity of the time derivative inside $\cC\cap\cI$. In the interior of $\cS$ we have that $v=g$ and thus $v_t=\dot{g}$ for all $(t,x)\in\inter(\cS)$.

\begin{theorem}\label{thm:v_tC0}
The function $v_t$ is continuous at the boundary $\partial \cC$, i.e., for all $(t_0,x_0)\in\partial\cC$ and for any sequence $((t_n,x_n))_{n\in\N}\subset\cC$ such that $(t_n,x_n)\to(t_0,x_0)$ we have
\begin{align*}
\lim_{n\to\infty}v_{t}(t_n,x_n)=\dot{g}(t_0).
\end{align*}
Moreover, there exists $D_1>0$ such that
\begin{align}\label{eq:M_time}
|v_t(t,x)|\leq D_1(1+|x|^{2\vee p})
\end{align}
for all $(t,x)\in[0,T]\times\cO$ where $p$ comes from \eqref{eq:Theta_Lip_time}.
\end{theorem}

\begin{proof}
From Proposition \ref{prop:monotonicity}, we have that $v(t_n,x_n)-v(t_n-\eps,x_n)\leq g(t_n)-g(t_n-\eps)$ for every $n\in\N$ and every $\eps\in(0,t_n)$. Therefore, dividing by $\eps$ and letting $\eps\to 0$, we obtain $v_t(t_n,x_n)\leq \dot g(t_n)$ for every $n\in\N$ and so
$$\limsup_{n\to\infty} v_t(t_n,x_n)\leq \dot g(t_0).$$

Now consider an arbitrary $(t,x)\in\cC$. Let $\nu^*\in\cA_{t,x}$ be the optimal control for the game $v(t,x)$ defined in \eqref{eq:RSDE}. We extend the control $\nu^*$ after $T-t$ by $\nu_s^*=\nu_{T-t}^*$ for $s\in(T-t,T-t+\eps]$. Let $\tau_*^\eps\in\cT_{t-\eps}$ be optimal for $v(t-\eps,x)$ and let $\tau^\eps=\tau_*^\eps\!\wedge\!(T\!-\!t)\in\cT_t$. Then, we have
\begin{align*}
&v(t,x)\!-\!v(t\!-\!\eps,x)\\
&\geq \E_x\Big[\e^{-r\tau^\eps}g(t\!+\!\tau^\eps)\!-\!\e^{-r\tau_*^\eps}g(t\!-\!\eps\!+\!\tau_*^\eps)\!+\!\int_0^{\tau^\eps}\!\e^{-rs}h(t\!+\!s,X^{\nu^*}_s)\ud s\!-\!\int_0^{\tau^\eps_*}\!\e^{-rs}h(t\!-\!\eps\!+\!s,X^{\nu^*}_s)\ud s\Big]\\
&=g(t)\!-\!g(t\!-\!\eps)\!+\!\E_x\Big[\int_0^{\tau^\eps}\!\e^{-rs}\big(\Theta(t\!+\!s,X_s^{\nu^*})\!-\!\Theta(t\!-\!\eps\!+\!s,X_s^{\nu^*})\big)\ud s\!-\!\int_{\tau^\eps}^{\tau^\eps_*}\!\e^{-rs}\Theta(t\!-\!\eps\!+\!s,X_s^{\nu^*})\ud s\Big]
\end{align*}
where $\Theta$ was defined in \eqref{eq:Theta}, the integrals with respect to the total variation $|\nu^*|_s$ are equal when computed on the interval random $[0,\tau^\eps]$ and $0$ on the random interval $(\tau^\eps,\tau^\eps_*]$. Using \eqref{eq:Theta_Lip_time} in the inequality above, we obtain
\begin{align*}
v(t,x)\!-\!v(t\!-\!\eps,x)\geq g(t)\!-\!g(t\!-\!\eps)\!-\!\eps C\E\Big[\tau^\eps\!\!\!\sup_{0\leq s\leq T-t}\big(1\!+\!|X_{s}^{\nu^*}|^p\big)\Big]\!-\!\E\Big[\int_{\tau^\eps}^{\tau_*^\eps}\!\!\e^{-rs}\Theta(t\!-\!\eps\!+\!s,X_{s}^{\nu^*})\,\ud s\Big]\!.
\end{align*}
Recalling that $\tau^\eps=\tau^\eps_*\wedge(T-t)\leq \tau_*^\eps$, $\P$-a.s., and that for any fixed $p\in[1,\infty)$, there exists $K_p>1$ such that $\E[\sup_{0\leq s\leq T-t+\eps}(1+|X^{\nu^*}_s|^p)]<K_p(1+|x|^p)$ by construction of the optimal strategy $\nu^*$ (see Corollary \ref{cor:Lpestnu} and recalling $\nu_u^*=\nu_{T-t}^*$ for $u\ge T-t$), using H\"older's inequality, we obtain
\begin{align*}
&v(t,x)\!-\!v(t\!-\!\eps,x)\\
&\ge g(t)\!-\!g(t\!-\!\eps)\!-\!\eps C\Big(\E\Big[\sup_{0\leq s\leq T-t}(1\!+\!|X_{s}^{\nu^*}|^p)^2\Big]\E\big[\big(\tau^\eps_*\!\wedge\!(T\!-\!t)\big)^2\big]\Big)^{1/2}\!\!-\!\E\Big[\int_{\tau^\eps}^{\tau_*^\eps}\!\!\e^{-rs}\Theta(t\!-\!\eps\!+\!s,X_{s}^{\nu^*})\ud s\Big]\\
&\ge g(t)\!-\!g(t\!-\!\eps)\!-\!\eps C\Big(4K_{2p}(1+|x|^{2p})\E\big[\big(\tau^\eps_*\wedge(T\!-\!t)\big)^2\big]\Big)^{1/2}\!-\!\E\Big[\int_{\tau_*^\eps\wedge(T-t)}^{\tau_*^\eps}\!\!\e^{-rs}|\Theta(t\!-\!\eps\!+\!s,X_{s}^{\nu^*})|\ud s\Big],
\end{align*}
where we used $(1+|x|)^2\le 2(1+|x|^2)$ in the first square root with also $2(1+K_{2p}(1+|x|^{2p}))\le 4K_{2p}(1+|x|^{2p})$ and we restored the definition of $\tau^\eps$.

We recall the convergence of $\tau_*^\eps\to\tau_*$ as $\eps\to0$ by Proposition \ref{prop:tauconv}. Noticing that the size of the random interval $[\tau_*^\eps\wedge(T-t),\tau_*^\eps]$ is smaller than $\eps$, for all $\omega\in\Omega$, dividing by $\eps$ the inequality above and sending $\eps\to0$, we have that
\begin{align*}
v_t(t,x)\geq \dot{g}(t)-C\Big(4K_{2p}(1+|x|^{2p})\E\big[\tau_*^2\big]\Big)^{1/2}-\E\big[\mathds{1}_{\{\tau_*=T-t\}}|\Theta(T,X_{T-t}^{\nu^*})|\big].
\end{align*}
Using H\"older's inequality in the last term above and noticing that $\Theta(t,x)^2\leq C'(1+|x|^4)$ by Assumption \ref{ass:1} for a suitable $C'>0$, we obtain
\begin{align}\label{eq:convvt0}
v_t(t,x)\geq \dot{g}(t)-C\Big(4K_{2p}(1+|x|^{2p})\E\big[\tau_*^2\big]\Big)^{1/2}-\Big(C'K_4(1+|x|^4)\E\big[\mathds{1}_{\{\tau_*=T-t\}}\big]\Big)^{1/2},
\end{align}
where $K_4$ comes from Corollary \ref{cor:Lpestnu}. 

The latter inequality holds, in particular, at $(t_n,x_n)\in\cC$ for every $n\in\N$. Since $\tau_*(\nu^*;t_n,x_n)\to\tau_*(\nu^*;t_0,x_0)= 0$, $\P$-a.s.\ by Proposition \ref{prop:tauconv}, we obtain 
\[
\liminf_{n\to\infty}v_t(t_n,x_n)\ge \dot{g}(t).
\]
Hence, $v_t$ is continuous at the boundary $\partial\cC$.

To conclude the proof, we show \eqref{eq:M_time}. Using $\tau_*\le T$, $\P$-a.s. and taking the minimum of $\dot{g}$, we have from \eqref{eq:convvt0}
\begin{align*}
v_t(t,x)\geq \min_{t\in[0,T]}\dot{g}(t)-2TC\sqrt{K_{2p}}(1+|x|^{p})-2\sqrt{C'K_4}(1+|x|^2).
\end{align*}
Recalling that $ v_t(t,x)\le \dot{g}(t)$ and taking the absolute value, we obtain
\begin{align*}
|v_t(t,x)|\le \max_{t\in[0,T]}|\dot{g}(t)| +2TC\sqrt{K_{2p}}(1+|x|^{p})+2\sqrt{C'K_4}(1+|x|^2)\le D_1(1+|x|^{2\vee p}),
\end{align*}
for all $(t,x)\in[0,T]\times\cO$ and some $D_1>0$.
\end{proof}

\section{Smooth-fit and convergence of optimal stopping times for the auxiliary problem}\label{sec:convergence}

In this section, we show some properties on the regularity of the value function of the auxiliary optimal stopping problem \eqref{eq:optstop} and on the regularity of the underlying process with respect to the free-boundary $b$. In particular, we prove that the smooth-fit condition at $b$ holds and that the optimal stopping times $\sigma^*(s,y)$ converge to $\sigma^*(t,x)$, $\P$-a.s., as $(s,y)\to (t,x)\in [0,T]\times\cO$.

The first lemma concerns the local boundedness of the second and third spatial derivatives and leads to the smooth-fit condition (Proposition \ref{prop:limvxxb-}).

\begin{lemma}\label{lem:vxxxbnd}
For every $t<T$, there exist $\eps>0$, $D_2=D_2(t,\eps)>0$ and $D_3=D_3(t,\eps)>0$ such that $|v_{xx}(s,x)|\leq D_2$ and $v_{xxx}(s,x)> -D_3$, for all $s\in[0,t]$ and $x\in[b(s)-\eps,b(s)+\eps]\setminus\{b(s)\}$.
\end{lemma}

\begin{proof}
Fix $t<T$. Recalling $a(t)<b(t)<\infty$, by combining \eqref{eq:a<b} and Proposition \ref{prop:monotone}(ii), there exists $\eps>0$ such that $B_{t,\eps}\coloneqq \{(s,x)\in[0,t]\times\cO:x\in(b(s)-\eps,b(s))\}\subseteq \cC\cap\cI$. For $x>b(s)$, we have $v_x(s,x)= \alpha_0$ and so $v_{xx}(s,x)=0$. Below $b(s)$, the value function $v$ satisfies the PDE in \eqref{eq:PDE} and, rewriting the equation, we obtain
\begin{align*}
v_{xx}(s,x)=\tfrac{2}{\sigma^2(x)}\Big[-v_t(s,x)-\mu(x)v_x(s,x)+rv(s,x)-h(s,x)\Big]
\end{align*}
for $x\in(b(s)-\eps,b(s))$. Since $\sigma^2(x)>0$ on $\cO$ by Assumption \ref{ass:1}, the functions $\mu$, $v$, $v_x$, $h$ are continuous and $v_t$ is locally bounded by Theorem \ref{thm:v_tC0}, we obtain $v_{xx}(s,x)\le D_2$ for all $(s,x)\in \bar{B}_{t,\eps}= [0,t]\times [b(s)-\eps,b(s)]$ (where we understand $v_{xx}(s,b(s))\le D_2$ as $\limsup_{n\to\infty} v_{xx}(s_n,x_n)\le D_2$ with $(s_n,x_n)\to (s,b(s))$ and $((s_n,x_n))_{n\in\N}\subset \cC\cap\cI$). Recalling that $v$ is convex by Proposition \ref{prop:monotonicity}(i), $v_{xx}$ is non-negative and so $|v_{xx}(s,x)|\leq D_2$ for all $s<t$ and $x\in[b(s)-\eps,b(s)+\eps]\setminus\{b(s)\}$.

Similarly, we obtain a lower bound for $v_{xxx}$. Since $v_x\equiv \alpha_0$ above $b$, it follows that $v_{xxx}(s,x)=0$ for $x>b(s)$. Below the boundary $b$, the function $v_x$ satisfies \eqref{eq:PDE_vx} and rewriting that PDE yields
\begin{align*}
v_{xxx}(s,x)=&\,\tfrac{2}{\sigma^2(x)}\Big[\!-\!v_{tx}(s,x)\!-\!\big(\sigma(x)\sigma_x(x)\!+\!\mu(x)\big)v_{xx}(s,x)\!+\!\big(r\!-\!\mu_x(x)\big)v_x(s,x)\!-\!h_x(s,x)\Big]\\
\ge&\,\tfrac{2}{\sigma^2(x)}\Big[\!-\!\big(\sigma(x)\sigma_x(x)\!+\!\mu(x)\big)v_{xx}(s,x)\!+\!\big(r\!-\!\mu_x(x)\big)v_x(s,x)\!-\!h_x(s,x)\Big]
\end{align*}
where we used that $v_{tx}$ is non-positive, in the interior of $\cC\cap\cI$, by Proposition \ref{prop:monotonicity}(iii). Since $\sigma$, $\sigma_x$, $\mu$, $\mu_x$, $h_x$ are continuous by Assumption \ref{ass:1}, $v$, $v_x$ are continuous by \eqref{eq:reg_v_vx} and $v_{xx}$ is locally bounded, then there exists $D_3>0$ such that $v_{xxx}(s,x)\ge-D_3$ for all $s<t$ and $x\in[b(s)-\eps,b(s)+\eps]\setminus\{b(s)\}$ (where we understand $v_{xxx}(s,b(s))\ge -D_3$ as $\liminf_{n\to\infty} v_{xxx}(s_n,x_n)\ge- D_3$ with $(s_n,x_n)\to (s,b(s))$ and $((s_n,x_n))_{n\in\N}\subset \cC\cap\cI$).
\end{proof}

The lower bound on $v_{xxx}$ allows the use of It\^o-Tanaka's formula which leads to the following smooth-fit condition at $b$. Here, we are inspired by estimates obtained in \cite[Lem.\ 11]{deangelis2021assets}.

\begin{proposition}\label{prop:limvxxb-}
For every $t<T$, we have
\begin{align*}
v_{xx}(t,b(t)-)\coloneqq \lim_{\eps\to0}\frac{v_x(t,b(t))-v_x(t,b(t)-\eps)}{\eps}= 0.
\end{align*}
\end{proposition}
\begin{proof}
Fix $t\in[0,T)$, the function $v_{xx}$ is continuous inside $\cC\cap\cI$ by classical interior regularity results for parabolic PDEs \cite[Th.\ 3.10]{friedman2008partial}. By convexity of $v$ we have
\begin{align*}
\lim_{\eps\to0}\frac{v_x(t,b(t))-v_x(t,b(t)-\eps)}{\eps}\geq 0.
\end{align*}
We argue by contradiction to show that the inequality above is an equality. Assume
\begin{align*}
v_{xx}(t,b(t)-)\coloneqq\lim_{\eps\to0}\frac{v_x(t,x)-v_x(t,x-\eps)}{\eps}>0.
\end{align*}
Fix $\eta\in(0,1)$ and introduce the stopping time 
\begin{align*}
\rho=\rho(\eta;t,b(t))=\inf\{u\geq 0:Y^{b(t)}_{u}\notin[b(t)-\eta,b(t)+\eta]\}
\wedge (T-t).
\end{align*}
Moreover, fix an $\bar\eps\in (0,T-t)$ such that $\P(\rho>\bar\eps)>0$ and let $\eps\in(0,\bar\eps)$ such that $b(t)-\eta>a(t+\eps)$. The latter can be obtained by continuity of $a$ and the fact that $b(t)>a(t)$ (cf.\ \eqref{eq:a<b}). It follows that $(s,Y_s^{b(t)})$ lies in $[0,\eps]\!\times\![b(t)\!-\!\eta,b(t)\!+\!\eta]\!\subseteq \![0,T\!-\!t]\!\times\! [b(t)\!-\!1,b(t)\!+\!1]\!\eqqcolon\! \cK $ for $s\in[0,\rho\wedge\eps]$. Applying Dynkin's formula on the random interval $[0,\rho\wedge\eps]$ and using \eqref{eq:ineqPDE_vx} yield
\begin{align*}
v_x(t,b(t))\leq&\,\E\Big[\e^{-\int_0^{\eps\wedge\rho}\lambda(Y_s^{b(t)})\,\ud s}v_x(t+\eps\wedge\rho,Y^{b(t)}_{\eps\wedge\rho})+\int_0^{\eps\wedge\rho}\!\e^{-\int_0^{s}\lambda(Y_u^{b(t)})\,\ud u}h_x(t+s,Y^{b(t)}_{s})\,\ud s\Big]\\
\le&\,\E\Big[\e^{-\int_0^{\eps\wedge\rho}\lambda(Y_s^{b(t)})\,\ud s}v_x(t,Y^{b(t)}_{\eps\wedge\rho})+(\eps\wedge\rho)\hat{C}_1\Big]\\
\le&\,\E\Big[\e^{-\int_0^{\eps\wedge\rho}\lambda(Y_s^{b(t)})\,\ud s}v_x(t,Y^{b(t)}_{\eps\wedge\rho})\Big]+\eps \,\hat{C}_1
\end{align*}
where the second inequality is justified by using that $v_x$ is non-increasing in time (recall Proposition \ref{prop:monotonicity}) and the use of $\hat{C}_1\coloneqq \sup_{(s,y)\in\cK}\exp(s \mu_x(y)-rs)h_x(t+s,y)$, where we recall the definition of $\lambda$ in \eqref{eq:lambda}.

Let $(L_s^{b(t)}(Y^{b(t)}))_{s\in[0,T-t]}$ be the semi-martingale local time at $b(t)$ of the process $Y^{b(t)}$ (see \cite[Ch.\ IV]{protter2005stochastic}). From an application of It\^o-Tanaka's formula to $\exp(-\int_0^{\eps\wedge\rho}\lambda(Y_s^{b(t)})\,\ud s)v_x(t,Y^{b(t)}_{\eps\wedge\rho})$ (see, e.g., \cite[Sec.\ 3.5]{peskir2006optimal} justified because $v_{xxx}$ is locally bounded on the interval $[b(t)-\eta,b(t)+\eta]\setminus\{b(t)\}$ by Lemma \ref{lem:vxxxbnd}), we have
\begin{align*}
v_x(t,b(t))\leq \E\Big[&\,v_x(t,b(t))+\frac{1}{2}\int_0^{\eps\wedge\rho}\!\e^{-\int_0^{s}\lambda(Y_u^{b(t)})\,\ud u}\big(v_{xx}(t,b(t)+)-v_{xx}(t,b(t)-)\big)\,\ud L_s^{b(t)}(Y^{b(t)})\\
&+\int_0^{\eps\wedge\rho}\!\e^{-\int_0^{s}\lambda(Y_u^{b(t)})\,\ud u}\Big((\sigma\sigma_x+\mu)v_{xx}-(r-\mu_x)v_x\Big)(t,Y^{b(t)}_{s})\,\ud s\\
&+\frac{1}{2}\int_0^{\eps\wedge\rho}\!\mathds{1}_{\{Y_s^{b(t)}\neq\, b(t)\}}\e^{-\int_0^{s}\lambda(Y_u^{b(t)})\,\ud u}\sigma^2(Y^{b(t)}_{s})v_{xxx}(t,Y^{b(t)}_{s})\,\ud s\Big]+\eps\, \hat{C}_1.
\end{align*}
Clearly $v_{xx}(t,x)=0$ for $x>b(t)$ because $v_x\equiv \alpha_0$ in $\cM$ and it follows $v_{xx}(t,b(t)+)=0$. Therefore,
\begin{align*}
0&\leq -\frac{1}{2}v_{xx}(t,b(t)-)\E\Big[\,\int_0^{\eps\wedge\rho}\!\e^{-\int_0^{s}\lambda(Y_u^{b(t)})\,\ud u}\,\ud L_s^{b(t)}(Y^{b(t)})\Big]\\
&\hspace{12pt}+\E\Big[\int_0^{\eps\wedge\rho}\!\e^{-\int_0^{s}\lambda(Y_u^{b(t)})\,\ud u}\Big((\sigma\sigma_x+\mu)v_{xx}-(r-\mu_x)v_x\Big)(t,Y^{b(t)}_{s})\,\ud s\Big]\\
&\hspace{12pt}+\frac{1}{2}\E\Big[\int_0^{\eps\wedge\rho}\!\mathds{1}_{\{Y_s^{b(t)}\neq\, b(t)\}}\e^{-\int_0^{s}\lambda(Y_u^{b(t)})\,\ud u}\sigma^2(Y^{b(t)}_{s})v_{xxx}(t,Y^{b(t)}_{s})\,\ud s\Big]+\eps\, \hat{C}_1.
\end{align*}

Recalling that $(Y_s^{b(t)})_{s\in[0,\rho\wedge\eps]}\subseteq [b(t)-1,b(t)+1]$ on the random interval $[0,\rho\wedge\eps]$, and $\lambda$, $\sigma$, $\sigma_x$, $\mu$, $\mu_x$, $v_x$, $v_{xx}$, $v_{xxx}$ are locally bounded, there exist constants $\hat{C}_2>0$ and $\hat{C}_3>0$ independent of $\rho$ such that
\begin{align*}
0\leq -v_{xx}(t,b(t)-)\hat{C}_2\E\big[L_{\eps\wedge\rho}^{b(t)}(Y^{b(t)})\big]+\hat{C}_3 \E\big[\eps\wedge\rho\big].
\end{align*}
Rearranging the above terms, we obtain
\begin{align}\label{eq:localtimebnd}
v_{xx}(t,b(t)-)\hat{C}_2\E\big[L_{\eps\wedge\rho}^{b(t)}(Y^{b(t)})\big]\leq \hat{C}_3 \E\big[\eps\wedge\rho\big]\le \eps\,\hat{C}_3.
\end{align}
Now the idea is to show a contradiction by finding a positive lower bound for the left-hand side. By Meyer-Tanaka's formula \cite[Cor.\ 3 of Th.\ IV.70]{protter2005stochastic} and taking expectation (notice that the stochastic integral therein is a martingale) we have that
\begin{align}\label{eq:localtimeY}
\begin{aligned}
\E\big[L_{\eps\wedge\rho}^{b(t)}(Y^{b(t)})\big]&=\E\Big[|Y^{b(t)}_{\eps\wedge\rho}-b(t)|-|Y_0^{b(t)}-b(t)|-\int_0^{\eps\wedge \rho}\!\sign(Y_s^{b(t)})\,\ud Y_s^{b(t)}\Big]\\
&=\E\Big[|Y^{b(t)}_{\eps\wedge\rho}-b(t)|-\int_0^{\eps\wedge \rho}\!\sign(Y_s^{b(t)})\big(\sigma(Y_s^{b(t)})\sigma_x(Y_s^{b(t)})+\mu(Y_s^{b(t)})\big)\,\ud s\Big].
\end{aligned}
\end{align}
Letting $p\in(0,1)$ and $\eta>0$ as above, considering the first term on the right-hand side above we have
\begin{align*}
\E\big[|Y^{b(t)}_{\eps\wedge\rho}-b(t)|\big]\ge \frac{1}{(2\eta)^p}\E\big[|Y^{b(t)}_{\eps\wedge\rho}-b(t)|^{1+p}\big],
\end{align*}
because $Y_{\eps\wedge\rho}\in[b(t)-\eta,b(t)+\eta]$. Using \eqref{eq:SDE_Y} and the inequality $|x-y|^{1+p}\ge \frac{1}{2^p}|x|^{1+p}-|y|^{1+p}$ (which holds by a combination of the triangle inequality and Jensen's inequality), we obtain
\begin{align}\label{eq:localtimeY1}
\begin{aligned}
&\frac{1}{(2\eta)^p} \E\big[|Y^{b(t)}_{\eps\wedge\rho}-b(t)|^{1+p}\big]\\
&\ge\frac{1}{(4\eta)^p} \E\Big[\Big|\int_0^{\eps\wedge\rho}\sigma(Y^{b(t)}_s)\ud W_s\Big|^{1+p}\Big]-\frac{1}{(2\eta)^p}\E\Big[\Big|\int_0^{\eps\wedge\rho}(\mu+\sigma\sigma_x)(Y^{b(t)}_s)\ud s\Big|^{1+p}\Big]\\
&\ge\frac{1}{(2\eta)^p} \E\Big[\Big|\int_0^{\eps\wedge\rho}\sigma(Y^{b(t)}_s)\ud W_s\Big|^{1+p}\Big]-\frac{C_\eta^{p+1}}{(2\eta)^p}\E\Big[(\eps\wedge\rho)^{1+p}\Big],
\end{aligned}
\end{align}
where the second inequality uses the continuity of $\mu,\sigma,\sigma_x$ to find constant $C_\eta$. The first term on the right-hand side of \eqref{eq:localtimeY1} can be bound from below by Doob's maximal inequality (see \cite[Th.\ 1.3.8(iv)]{karatzas1998brownian} and Burkholder-Davis-Gundy inequality (see \cite[Ch.\ 14.18]{williams1991probability}) so that
\begin{align}\label{eq:localtimeY2}
\begin{aligned}
\E\Big[\Big|\int_0^{\eps\wedge\rho}\sigma(Y^{b(t)}_s)\ud W_s\Big|^{1+p}\Big]&\ge \Big(\frac{p}{1+p}\Big)^{1+p}\E\Big[\sup_{0\le t\le \eps\wedge\rho}\Big|\int_0^{t}\sigma(Y^{b(t)}_s)\ud W_s\Big|^{1+p}\Big]\\
&\ge C_p\Big(\frac{p}{1+p}\Big)^{1+p}\E\Big[\Big|\int_0^{\eps\wedge\rho}\sigma^2(Y^{b(t)}_s)\ud s\Big|^{\frac{1+p}{2}}\Big]\\
&\ge C_{\eta,p}\E\big[(\eps\wedge\rho)^{\frac{1+p}{2}}\big]
\end{aligned}
\end{align}
where $C_p>0$ is the constant appearing in the Burkholder-Davis-Gundy inequality and $C_{\eta,p}>0$ is a constant independent of $\eps$ (notice that $C_p$ depends on the minimum of the function $\sigma(y)$ in $[b(t)-1,b(t)+1]$).

Plugging \eqref{eq:localtimeY1} and \eqref{eq:localtimeY2} into \eqref{eq:localtimeY} yields
\begin{align*}
\E\big[L_{\eps\wedge\rho}^{b(t)}(Y^{b(t)})\big]&\ge \frac{C_{\eta,p}}{(2\eta)^p} \E\big[(\eps\wedge\rho)^{\frac{1+p}{2}}\big]-C_\eta\E\big[\eps\wedge\rho\big]-\frac{C_\eta^{p+1}}{(2\eta)^p}\E\big[(\eps\wedge\rho)^{1+p}\big].
\end{align*}
Assuming (without loss of generality) $\eps<1$, we can collect the last two terms above as follows
\begin{align*}
\E\big[L_{\eps\wedge\rho}^{b(t)}(Y^{b(t)})\big]&\ge \frac{C_{\eta,p}}{(2\eta)^p} \E\big[(\eps\wedge\rho)^{\frac{1+p}{2}}\big]-\Big(C_\eta+\frac{C_\eta^{p+1}}{(2\eta)^p}\Big)\eps.
\end{align*}
Thus, combining the inequality above with \eqref{eq:localtimebnd} and collecting terms of the same order, we have
\begin{align*}
v_{xx}(t,b(t)-)\hat{C}_2\frac{C_{\eta,p}}{(2\eta)^p}\E\big[(\eps\wedge\rho)^{\frac{1+p}{2}}\big]\leq \Big[\hat{C}_3+v_{xx}(t,b(t)-)\hat{C}_2\Big(C_\eta+\frac{C_\eta^{p+1}}{(2\eta)^p}\Big)\Big]\eps.
\end{align*}
Denoting $\tilde{C}_{\eta,p}\coloneqq \hat{C}_2\frac{C_{\eta,p}}{(2\eta)^p} $ and $\hat{C}_{\eta,p}\coloneqq\big[\hat{C}_3+v_{xx}(t,b(t)-)\hat{C}_2\big(C_\eta+\frac{C_\eta^{p+1}}{(2\eta)^p}\big)\big]$, and using that $v_{xx}(t,b(t)-)$ is bounded by Lemma \ref{lem:vxxxbnd}, yield
\begin{align*}
v_{xx}(t,b(t)-)\tilde{C}_{\eta,p}\E\big[(\eps\wedge\rho)^{\frac{1+p}{2}}\big]\le\hat{C}_{\eta,p} \eps.
\end{align*}
Moreover, we have $\E\big[(\eps\wedge\rho)^{\frac{1+p}{2}}\big]\ge \eps^{\frac{1+p}{2}}\P(\rho>\eps)\ge \eps^{\frac{1+p}{2}}\P(\rho>\bar\eps)$, recalling that $\eps<\bar\eps$. Therefore, we obtain
\begin{align*}
v_{xx}(t,b(t)-)\tilde{C}_{\eta,p}\eps^{\frac{1+p}{2}}\P(\rho>\bar\eps)\leq \hat{C}_{\eta,p} \eps,
\end{align*}
which implies
\begin{align*}
0<v_{xx}(t,b(t)-)\tilde{C}_{\eta,p}\P(\rho>\bar\eps)\le \hat{C}_{\eta,p} \eps^{\frac{1-p}{2}},
\end{align*}
where the first inequality follows by recalling that $\P(\rho>\bar\eps)>0$. Sending $\eps\to0$ and recalling that $p\in(0,1)$, we reach a contradiction because the right-hand side above goes to $0$. Thus, it holds $v_{xx}(t,b(t)-)=0$.
\end{proof}

To obtain our desired results, we need to introduce two stopping times which are slightly different from $\sigma_*$ and prove some properties that they satisfy. That is, for fixed $(t,y)\in[0,T]\times\cO$ and recalling $Y^y$ from \eqref{eq:SDE_Y}, let us define the stopping times $\hat{\sigma}=\hat{\sigma}(t,y)$ and $\check{\sigma}=\check{\sigma}(t,y)$ as
\begin{align*}
\begin{aligned}
\hat{\sigma}=\inf\{s>0:Y_s^y\geq b(t+s)\}\wedge(T-t) \quad \text{and} \quad 
\check{\sigma}=\inf\{s>0:Y_s^y> b(t+s)\}\wedge(T-t).
\end{aligned}
\end{align*}
The next result shows that $\hat{\sigma}$ and $\sigma_*$ from \eqref{eq:sigma*} are equal, $\P$-a.s. We will show later that $\check{\sigma}$ is equal to the other two $\P$-a.s.

\begin{lemma}\label{lem:sigma*hat}
For all $(t,y)\in[0,T]\times \cO$, we have $\sigma_*(t,y)=\hat{\sigma}(t,y)$, $\P$-a.s. Moreover, for $(t,y)\in\cM$, it holds $\P(\hat{\sigma}(t,y)=0)=1$.
\end{lemma}
\begin{proof}
Recalling that $[0,T]\times \cO$ can be split in two disjoint regions $\cM$ and $\cI$, we divide the proof by considering three cases: $(t_0,y_0)\in\inter(\cM)$, $(t_0,y_0)\in\cI$ and $(t_0,y_0)\in\partial\cI$. 

{\bf Case 1:} $(t_0,y_0)\in\inter(\cM)$. We have that $\sigma_*(t_0,y_0)=0$, $\P$-a.s and using that $Y^{y_0}$ and $b$ are continuous functions (the process for almost every $\omega)$, we have that also $\hat{\sigma}(t_0,y_0)=0$. Thus, the result holds.

{\bf Case 2:} $(t_0,y_0)\in\cI$. If $t_0=T$, then $\sigma_*=\hat{\sigma}=0$ by definition. So let $t_0<T$. Since the process is continuous almost every $\omega$ and it is starting from an internal point $y_0<b(t_0)$, we have that 
\begin{align}\label{eq:sigmahat}
\sigma_*(t_0,y_0)=\inf\{s\ge 0:Y_s^{y_0}\geq b(t_0+s)\}=\inf\{s>0:Y_s^{y_0}\geq b(t_0+s)\}=\hat{\sigma}(t_0,y_0).
\end{align}

{\bf Case 3:} $(t_0,y_0)\in\partial\cI$. The process $Y$ is a continuous Feller process, so by the Blumenthal's zero-one law (cf. \cite[Ch. III.9]{rogers2000diffusions1}) we have either $\P(\hat{\sigma}>0)=1$ or $\P(\hat{\sigma}>0)=0$. Assume by contradiction that $\P(\hat{\sigma}>0)=1$. Let $y<y_0$ and fix $\eps>0$ such that $y+\eps<y_0$. Let us define $\sigma^\eps\coloneqq \hat{\sigma}(t,y+\eps)$ and notice that by {\bf Case 2} it is such that $\sigma^\eps=\sigma_*(t,y+\eps)$, $\P$-a.s. Recalling that $v_x$ admits the probabilistic representation \eqref{eq:optstop} and denote $\tau_a$ by $\tau_a^y$ to keep track on its dependence on the starting point $y$. Since $\sigma^\eps$ is sub-optimal for $v_x(t,y)$ and optimal for $v_x(t,y+\eps)$, then we have
\begin{align}\label{eq:v_xlowerbnd}
\begin{aligned}
&v_x(t,y\!+\!\eps)\!-\!v_x(t,y)\\
&\ge \E\Big[\mathds{1}_{\{\sigma^\eps< (T-t)\wedge\tau_a^{y+\eps}\}}\e^{-\int_0^{\sigma^\eps}\lambda(Y_u^{y+\eps})\ud u}\alpha_{0}\!+\!\int_{0}^{\sigma^\eps\wedge\tau_a^{y+\eps}}\!\!\!\!\e^{-\int_0^s\lambda(Y_u^{y+\eps})\ud u}h_x(t\!+\!s,Y_s^{y+\eps})\ud s\\
&\qquad-\mathds{1}_{\{\sigma^\eps< (T-t)\wedge\tau_a^y\}}\e^{-\int_0^{\sigma^\eps}\lambda(Y_u^y)\ud u}\alpha_{0}\!-\!\int_{0}^{\sigma^\eps\wedge\tau_a^y}\!\!\!\!\e^{-\int_0^s\lambda(Y_u^y)\ud u}h_x(t\!+\!s,Y_s^y)\,\ud s\Big].
\end{aligned}
\end{align}
Notice that by standard comparison principles we have $Y_s^{y+\eps}\geq Y_s^{y}$, $\P$-a.s for all $s\geq0$. Therefore, we have $\tau_a^{y+\eps}\geq \tau_a^{y}$, $\P$-a.s., $\mathds{1}_{\{\sigma^\eps< (T-t)\wedge\tau_a^{y+\eps}\}}\ge \mathds{1}_{\{\sigma^\eps< (T-t)\wedge\tau_a^{y}\}}$ and, using also $x\mapsto\lambda (x)$ is non-increasing (recall \eqref{eq:lambda} and Assumption \ref{ass:1}),
\begin{align*}
\e^{-\int_0^{\sigma^\eps}\lambda(Y_u^{y+\eps})\ud u}\ge \e^{-\int_0^{\sigma^\eps}\lambda(Y_u^{y})\ud u}. 
\end{align*}
Plugging it with the inequality above into \eqref{eq:v_xlowerbnd}, yields
\begin{align*}
v_x(t,y+\eps)-v_x(t,y)\ge &\,\E\Big[\int_{0}^{\sigma^\eps\wedge\tau_a^{y}}\!\e^{-\int_0^s\lambda(Y_u^{y})\ud u}\Big(h_x(t+s,Y_s^{y+\eps})-h_x(t+s,Y_s^y)\Big)\,\ud s\Big]
\end{align*}
where we also used that $h_x$ is non-negative. Dividing by $\eps$ and sending $\eps\to0$, we obtain
\begin{align*}
v_{xx}(t,y)\ge &\,\E\Big[\liminf_{\eps\to0}\frac{1}{\eps}\int_{0}^{\sigma^\eps\wedge\tau_a^{y}}\!\e^{-\int_0^s\lambda(Y_u^{y})\ud u}\Big(h_x(t+s,Y_s^{y+\eps})-h_x(t+s,Y_s^y)\Big)\,\ud s\Big]
\end{align*}
where the limit and expectation are exchange by Fatou's lemma and the fact $h_x(t\!+\!s,Y_s^{y+\eps})\!\geq\! h_x(t\!+\!s,Y_s^y)$ by convexity of $h$. The sequence $(\sigma^\eps)_{\eps\in(0,1)}$ is non-decreasing as $\eps\downarrow0$, $\P$-a.s. Thus it admits a limit and we denote it by $\sigma^{+}\coloneqq\lim_{\eps\downarrow0}\sigma^\eps$. We claim and we prove later at the end of the proof that $\sigma^{+}=\sigma^0$. Therefore,
\begin{align*}
v_{xx}(t,y)\ge &\,\E\Big[\int_{0}^{\sigma^0\wedge\tau_a^{y}}\!\e^{-\int_0^s\lambda(Y_u^{y})\ud u}\Big(h_{xx}(t+s,Y_s^{y})\partial_xY_s^y)\Big)\,\ud s\Big]
\end{align*}
where $\partial_x Y_s^y$ denotes the solution of the SDE associated to the derivative with respect to the initial position of the process $Y^y$. Using that $h_{xx}$ is positive on $\cO$ and that the process $\partial_x Y_s^y$ is strictly positive, we can introduce a localising stopping time $\rho^y=\rho^y(\delta)$ for fixed $\delta>0$ such that
\begin{align*}
\rho^y=\inf\{s\geq 0|Y^{y}_s\notin[y-\delta,y_0+\delta]\}\wedge\inf\{s\geq 0|\partial_xY^{y}\leq \delta\}\wedge(T-t).
\end{align*}
It follows that 
\begin{align*}
v_{xx}(t,y)\ge &\,\E\Big[\int_{0}^{\sigma^0\wedge\tau_a^{y}\wedge\rho^y}\!\e^{-\int_0^s\lambda(Y_u^{y})\ud u}\Big(h_{xx}(t+s,Y_s^{y})\partial_xY_s^y)\Big)\,\ud s\Big]\geq C_1\E[\sigma^0\wedge\tau_a^{y}\wedge\rho^y]
\end{align*}
where $C_1>0$ is the minimum of the function $\e^{-\int_0^s\lambda(Y_u^{y})\ud u}h_{xx}(t+s,Y_s^{y})\partial_x Y_s^{y}$ over the interval $[0,\sigma^+\wedge\tau_a^{y}\wedge\rho^y]$.

Sending $y\uparrow y_0$ we have that $\rho^y\to\rho^{y_0}$, $\tau_a^y\to\tau_a^{y_0}$ and $\lim_{y\uparrow y_0}\sigma^0(t,y)\geq \hat{\sigma}(t_0,y_0)$ because $y\mapsto \hat{\sigma}(t,y)$ is non-increasing (notice we wrote explicitly the dependence of $\sigma^0=\hat\sigma(t,y)$ on the pair $(t,y)$). It follows by dominated convergence that 
\begin{align*}
v_{xx}(t,y_0)\ge C_1\E[\hat{\sigma}\wedge\tau_a^{y_0}\wedge\rho^{y_0}].
\end{align*}
By assumption $\P(\hat{\sigma}>0)=1$, by continuity of the process both $\tau_a^{y_0}$ and $\rho^{y_0}$ are strictly positive $\P$-a.s, it turns out that we have a contradiction because we obtain $v_{xx}(t,y_0)>0$, but $v_{xx}(t,y_0)=0$ by Proposition \ref{prop:limvxxb-}.

{\bf Proof of $\sigma^{+}=\sigma^0$.} First notice that $\sigma^\eps\leq \sigma^0$, $\P$-a.s. for all $\eps>0$ and so $\sigma^+\leq \sigma^0$, $\P$-a.s. Let $\omega\in\Omega$ such that $\sigma^+(\omega)\leq \sigma^0(\omega)$. If $\sigma^0(\omega)=0$, then $\sigma^+(\omega)=0$. So, assume that $\sigma^0(\omega)>0$ and pick any $\delta>0$ such that $\sigma^0(\omega)>\delta$. Then, there exists $c_\delta$ such that
\begin{align*}
\inf_{0\leq s\leq \delta}\big[b(t+s)-Y^{y}_s(\omega)\big]\ge c_\delta(\omega)>0.
\end{align*}
Since $\lim_{\eps\to0}\sup_{0\leq s\leq \delta}|Y^{y+\eps}_s(\omega)- Y^y_s(\omega)|=0$ converges uniformly (up to selecting a subsequence), for all sufficiently small $\eps$, we have that $Y^{y+\eps}_s\leq Y^{y}_s+ \tfrac{c_\delta(\omega)}{2}$ for all $s\in[0,\delta]$. Thus
\begin{align*}
\inf_{0\leq s\leq \delta}\big[b(t+s)-Y^{y+\eps}_s(\omega)\big]\ge \tfrac{c_\delta(\omega)}{2}>0.
\end{align*}
It follows that $\liminf_{\eps\downarrow0}\sigma^\eps(\omega)\geq \delta$. By the arbitrariness of $\delta$ we have $\liminf_{\eps\downarrow0}\sigma^\eps(\omega)\geq \sigma^0(\omega)$. So, $\sigma^+=\sigma^0$ holds for almost every $\omega$.
\end{proof}

In the next lemma we prove that $\sigma_*(t,y)=\check{\sigma}(t,y)$, $\P$-a.s., for every $(t,y)\in[0,T]\times \cO$. Since $b$ is an upper boundary, this result would follow by a standard application of the law of iterated logarithm if $b$ were non-increasing. That is, if $b$ were non-increasing, by the law of iterated logarithm we would have that the first time that $Y$ hits $b$ would coincide the first time that $Y$ crosses $b$. However, $b$ is non-decreasing in our framework and so the proof of this result requires finer technicalities. To this purpose, we mimic techniques that were originally developed in \cite{cox2015embedding} and then adopted in other works (see, e.g., \cite[Lem.\ 10]{deangelis2021assets}).

\begin{lemma}\label{lem:sigma*=check-sigma}
For all $(t,y)\in[0,T]\times \cO$, we have that $\sigma_*(t,y)=\check{\sigma}(t,y)$, $\P$-a.s.
\end{lemma}
\begin{proof}
By Lemma \ref{lem:sigma*hat} it is enough to show that $\hat{\sigma}(t,y)=\check{\sigma}(t,y)$, $\P$-a.s., for any $(t,y)\in[0,T]\times \cO$.
Recall the form of $\cI$ from \eqref{eq:C-I-boundaries} and that $\cM=\cI^c$. Then, it is straightforward to see that $\hat{\sigma}(t,y)=\check{\sigma}(t,y)=0$ if $(t,y)\in\mathring{\cM}\cup(\{T\}\times\cO)$.

So let us fix $(t,y)\in\bar{\cI}\cap ([0,T)\times\cO)$, i.e., $t<T$ and $y\leq b(t)$. For $\eps,\delta>0$, we define
\begin{align*}
	\hat{\sigma}_\eps\coloneqq\inf\{s>0|Y^y_s\geq b(t+s)+\eps \}, \quad &\hat{\sigma}^\delta_\eps\coloneqq\inf\{s>\delta|Y^y_s\geq b(t+s)+\eps \},\\
	\check{\sigma}_\eps\coloneqq\inf\{s>0|Y^y_s> b(t+s)+\eps \}, \quad &\check{\sigma}^\delta_\eps\coloneqq\inf\{s>\delta|Y^y_s> b(t+s)+\eps \},
\end{align*}
so that $\hat{\sigma}=\hat{\sigma}_0$ and $\check{\sigma}=\check\sigma_0$. By continuity of $Y$ and $b(\cdot)$ and monotonicity of $\eps\mapsto \hat{\sigma}^\delta_\eps$, we have that 
\begin{equation}\label{eq:lim_eps-sigma}
	\lim_{\eps\to 0}\hat{\sigma}_\eps=\check\sigma \quad \text{and} \quad \hat{\sigma}_0^\delta\leq \lim_{\eps\to 0}\hat{\sigma}_\eps^\delta=\check{\sigma}^\delta_0, \qquad \P\text{-a.s.}
\end{equation}
Assume that, for any $(t,y)\in\bar{\cI}$, it holds
\begin{equation}\label{eq:sigma>s}
	\P(\check\sigma_0^\delta>s)\leq \P(\hat\sigma_0^\delta>s), \qquad \forall \: s\geq 0.
\end{equation}
Then, by \eqref{eq:lim_eps-sigma} and \eqref{eq:sigma>s}, we obtain $\check\sigma^\delta_0=\hat\sigma^\delta_0$, $\P$-a.s.\ and so
\begin{equation*}
	\check\sigma =\lim_{\eps\to 0}\hat\sigma_\eps=\lim_{\eps\to 0}\lim_{\delta\to 0} \hat\sigma^\delta_\eps=\lim_{\delta\to 0}\lim_{\eps\to 0} \hat\sigma^\delta_\eps=\lim_{\delta\to 0}\check\sigma_0^\delta=\lim_{\delta\to 0}\hat\sigma^\delta_0=\hat\sigma_0=\hat\sigma,
\end{equation*}
where we have used \eqref{eq:lim_eps-sigma} and the exchange of limits is justified by the fact that $\hat\sigma^\delta_\eps$ is non-decreasing in both $\eps$ and $\delta$.

It remains to prove \eqref{eq:sigma>s}. Fix $(t,y)\in\bar{\cI}\cap ([0,T)\times\cO)$ and let $b^t(s)\coloneqq b(t+s)$. Notice that, since $b^t$ is non-decreasing, any interval of the form $(\delta,s)$ may be decomposed into the union of countably many intervals where $b^t$ is either flat or strictly increasing. Let $F$ be the set (countable union of intervals) where $b^t$ is flat. Then, it follows from the law of iterated logarithm that $\hat\sigma=\check\sigma$ on the event $\{\hat\sigma\in F \}$, because $Y$ immediately crosses $b^t$ if it hits it on $F$. As a consequence, we have that
\begin{equation}\label{eq:PonF}
	\P(Y_s\leq b^t(s),\forall\: s\in F )=\P(Y_s< b^t(s),\forall\: s\in F ).
\end{equation}
Fix $0\leq\delta<s<T-t$ and $h_0\in(0,s-\delta)$. Then, since $b^t$ is non-decreasing, we have that $b^t(s)\leq b^t(s+h)$ for every $h\in(0,h_0)$. Moreover, the latter inequality is strict whenever $b^t$ is strictly increasing. By this last argument and \eqref{eq:PonF}, for any $h\in(0,h_0)$, we obtain
\begin{align*}
\P(\check\sigma^\delta_0>s)&=\P(Y_u\leq b^t(u),\forall \: u\in(\delta,s])\leq \P(Y_u< b^t(u+h),\forall \: u\in(\delta,s])\\
&=\P(Y_{u-h}< b^t(u),\forall \: u\in(\delta+h,s+h])\leq \P(Y_{u-h}< b^t(u),\forall \: u\in(\delta+h_0,s]).
\end{align*}
By using the notation $\P_y(\,\cdot\,)=\P(\,\cdot\,|\,Y_0=y)$, the Markov property implies
\begin{align}\label{eq:speedmeas}
\P_y(\check\sigma_0^\delta>s)&\leq \P_y(Y_{u-h}< b^t(u),\forall \: u\in(\delta+h_0,s])\nonumber\\
&=\E_y\big[\P_{Y_{h_0+\delta/2-h}}(Y_{u-h_0-\delta/2}< b^t(u),\forall \: u\in(\delta+h_0,s]) \big]\\
&=\int_{\inf\cO}^{\infty} p_Y(\xi;h_0\!+\!\delta/2\!-\!h,y)\P_\xi\big(Y_{u-h_0-\delta/2}< b^t(u),\forall \: u\in(\delta+h_0,s]\big)\,m_Y(\ud\xi),
\end{align}
where $p_Y$ is a continuous function and $m_Y$ is the speed measure of $Y$ (see, e.g., \cite[Th.\ 50.11]{rogers2000diffusions}). By Scheffe's theorem (see, e.g., \cite[Th.\ 16.12]{billingsley2013convergence}), we have that
\begin{align*}
\lim_{h\to 0}\int_{\inf\cO}^{\infty} \big|\,p_Y(\xi;h_0\!+\!\delta/2\!-\!h,y)-p_Y (\xi;h_0\!+\!\delta/2,y)\big|\,m_Y(\ud\xi)=0,
\end{align*}
and so, letting $h\to 0$ in \eqref{eq:speedmeas}, we obtain
\begin{align*}
\P_y(\check\sigma_0^\delta>s)&\leq\int_{\inf\cO}^{\infty} p_Y(\xi;h_0\!+\!\delta/2,y)\P_\xi\big(Y_{u-h_0-\delta/2}< b^t(u),\forall \: u\in(\delta+h_0,s]\big)\,m_Y(\ud\xi)\\
&=\P_y(Y_{u}< b^t(u),\forall \: u\in(\delta+h_0,s])=\P_y(\hat{\sigma}^{\delta+h_0}_0>s).
\end{align*}
Letting $h_0\to 0$, we reach our desired result \eqref{eq:sigma>s}, because $\hat{\sigma}^{\delta+h_0}_0\downarrow \hat{\sigma}^{\delta}_0$ as $h_0\to 0$.
\end{proof}

In the next proposition, we prove the convergence of stopping times $\sigma_*(t_n,y_n)\to \sigma_*(t,y)$ for any sequence $((t_n,y_n))_{n\in\N}$ such that $(t_n,y_n)\to(t,y)$.

\begin{proposition}\label{prop:sigma_n-to-sigma_*}
Let $(t,y)\in[0,T]\times\cO$, for any sequence $((t_n,y_n))_{n\in\N}\subset[0,T]\times\cO$ such that $(t_n,y_n)\to(t,y)$, then $\sigma_*(t_n,y_n)\to \sigma_*(t,y)$, $\P$-a.s.
\end{proposition}

\begin{proof}
To simplify the notation, we let $\sigma_*=\sigma_*(t,y)$, $\sigma^n_*=\sigma_*(t_n,y_n)$, $\check\sigma=\check\sigma(t,y)$ and $\check\sigma^n=\check\sigma(t_n,y_n)$ for every $n\in\N$. Let $\Omega_n\coloneqq\{\sigma^n_*=\check\sigma^n \}$, $\Omega_0\coloneqq\{\sigma_*=\check\sigma\}$ and $\bar\Omega\coloneqq\cap_{n\in\N\cup\{0\}}\Omega_n\in\cF$. Then, by Lemma \ref{lem:sigma*=check-sigma}, we have that $\P(\bar\Omega)=1$. In the proof we consider $\omega\in\bar\Omega$ fixed but we omit it to simplify the notation. If $t=T$, the claim is easy to prove as $\sigma_*=0$ and so $|\sigma^n_*-\sigma_*|\leq T-t_n\to 0$ as $n\to\infty$.

Let $t<T$. In this case, we show the desired result $\lim_{n\to\infty}\sigma_*^n=\sigma_*$ by first showing that $\limsup_{n\to\infty}\sigma_*^n\leq\sigma_*$ and then showing that $\liminf_{n\to\infty}\sigma_*^n\geq\sigma_*$.

To prove the first inequality, we have that either $Y^y_s< b(t+s)$ for every $s\in [0,T-t)$ or there exists $s\in [0,T-t)$ such that $Y^y_s\geq b(t+s)$. In the former case, we obtain $\sigma_*=T-t$ and so
$$\limsup_{n\to\infty}\sigma^n_*\leq\limsup_{n\to\infty}(T-t_n)=T-t=\sigma_*.$$
In the latter case, by Lemma \ref{lem:sigma*=check-sigma}, there exists $s'\in[0,T-t)$ such that $Y^y_{s'}>b(t+s')$. Then, continuity of $y\mapsto Y^y_{s'}$ and of $t\mapsto b(t+s')$ yield $Y^{y_n}_{s'}>b(t_n+s')$ for $n$ sufficiently large. Hence, $\check\sigma^n\leq s'$ and so $\limsup_{n\to\infty} \check\sigma^n\leq s'$. Since the last inequality holds for any $s'\in [0,T-t)$ such that $Y^y_{s'}>b(t+s')$, we also obtain $\limsup_{n\to\infty} \check\sigma^n\leq \check\sigma$. Therefore, by Lemma \ref{lem:sigma*=check-sigma}, we have that
$$\limsup_{n\to\infty} \sigma_*^n=\limsup_{n\to\infty} \check\sigma^n\leq \check\sigma=\sigma_*$$
and the first inequality is proved.

To show the second inequality, first assume that $y\geq b(t)$. Then, $\sigma_*=0$ and so $$\liminf_{n\to\infty}\sigma^n_*\geq 0=\sigma_*.$$
Now let $y<b(t)$ so that $\sigma_*>0$, $\P$-a.s. By continuity of $u\mapsto b(t+u)$ and $u\mapsto Y^y_u$, let $s\in(0, \sigma_*)$ be such that $Y^y_u<b(t+u)$ for every $u\in [0,s]$. In particular, we have that
$$\inf_{u\in[0,s]}\big(b(t+u)-Y^y_u \big)\eqqcolon 2\delta>0.$$
Then, by continuity of $y\mapsto Y^y_u$ uniformly in $u\in [0,s]$ and of $t\mapsto b(t+u)$, we obtain
$$\inf_{u\in[0,s]}\big(b(t_n+u)-Y^{y_n}_u \big)\geq \delta>0,$$
for $n$ sufficiently large. Hence, $\sigma^n_*\geq s$ and so $\liminf_{n\to\infty} \sigma^n_*\geq s$. Since $s\in (0, \sigma_*)$ such that $Y^y_u<b(t+u)$ for every $u\in [0,s]$ was arbitrary, we have that
$$\liminf_{n\to\infty} \sigma^n_*\geq \sigma_*,$$
which concludes the proof.
\end{proof}

\section{Differentiability of \texorpdfstring{$v_{x}$}{vₓ} at the boundary \texorpdfstring{$b$}{b}}\label{sec:vxdiff}
In this section, we study the differentiability of the function $v_{x}$ across the boundary $b$. In particular, we will show that both $v_{tx}$ and $v_{xx}$ are continuous across $b$ and, more generally, in $\cC$.

We first focus on the mixed derivative. We recall that $v_{x}\equiv \alpha_0$ inside $\cM$, i.e., $v_{tx}(s,y)=0$ for all $(s,y)\in\inter(\cM)$. By classical interior regularity results for parabolic PDEs \cite[Th.\ 3.10]{friedman2008partial} we have that $v_{tx}$ is continuous inside $\cC\cap\cI$. Thus, it remains to show the continuity of $v_{tx}$ at the boundary $b$. We start by proving that $t\mapsto v_x(t,y)$ is locally Lipschitz continuous.

\begin{proposition}\label{prop:v_x-Lip-time}
The function $t\mapsto v_{x}(t,y)$ is locally Lipschitz continuous. In particular, for every $(t,y)\in\cC\cap\cI$ and $\eps\in(0,t)$, we have that
$$|v_x(t,y)-v_x(t-\eps,y)|\leq D_4 (1+|y|^p)\eps.$$
for some constant $D_4>0$.
\end{proposition}

\begin{proof}
 Fix $(t,y)\in\cC\cap\cI$, so that $a(t)<y<b(t)$. By standard results in optimal stopping theory, it is well known that the following (sub-)martingale property holds for the stopping problem \eqref{eq:optstop}
 \begin{equation}\label{eq:v_xMart}
 v_x(t,y) \leq \E\Big[\int_0^\rho \e^{-\int_0^s\lambda(Y^y_u)\ud u}h_x(t+s,Y^y_s)\ud s+\e^{-\int_0^\rho \lambda(Y^y_s)\ud s} v_x(t+\rho,Y^y_\rho)\Big],
 \end{equation}
 for every stopping time $\rho\in[0,\tau_a(t,y)]$, with equality if $\rho\leq \sigma_*(t,y)$, $\P$-a.s.

Let $d:[0,T]\to\R$ be a continuous, non-decreasing boundary such that $a(s)<d(s)<b(s)$ for every $s\in[0,T]$ and such that $y>d(t)$. Such a boundary exists by \eqref{eq:a<b}. Accordingly, define
\begin{align}\label{eq:tau_d}
 \tau_d=\tau_d(t,y)\coloneqq\inf\{s\geq 0|Y^y_s\leq d(t+s) \}\wedge(T-t).
\end{align}
Fix $\eps\in(0,t)$. Since $a(t-\eps+s)<d(t-\eps+s)\leq d(t+s)$ for every $s\in[0, T-t)$, we have that
$$\tau_d(t,y)\leq \tau_d(t-\eps,y)\leq\tau_a(t-\eps,y), \quad \P\text{-a.s.}$$
Therefore, by \eqref{eq:v_xMart} and letting $\rho=\sigma_*(t,y)\wedge\tau_d(t,y)\wedge S$ for a fixed $S\in(0,T-t)$, we obtain
\begin{align}\label{eq:v_x-diff}
\begin{aligned}
v_x(t,y)-v_x(t-\eps,y) &\geq \E\Big[\int_0^\rho \e^{-\int_0^s\lambda(Y^y_u)\ud u}\big(h_x(t+s,Y^y_s)-h_x(t-\eps+s,Y^y_s)\big)\ud s\Big] \\
&\hspace{11pt}+\E\Big[\e^{-\int_0^\rho \lambda(Y^y_s)\ud s} \big(v_x(t+\rho,Y^y_\rho)-v_x(t-\eps+\rho,Y^y_\rho)\big)\Big].
\end{aligned}
\end{align}
For the first term on the right-hand side of \eqref{eq:v_x-diff}, by \eqref{eq:hx_Lip_time} and standard estimates for SDEs, we have that
\begin{align}\label{eq:h_x-diff}
\begin{aligned}
\E\Big[\int_0^\rho\! \e^{-\int_0^s\lambda(Y^y_u)\ud u}\big(h_x(t\!+\!s,Y^y_s)\!-\!h_x(t\!-\!\eps\!+\!s,Y^y_s)\big)\ud s\Big]&\ge-C_0\eps\E\Big[\rho \sup_{s\in[0,T-t]}\big(1\!+\!|Y_s^y|^p\big)\Big]\\
&\geq -C_1(1+|y|^p)\eps,
\end{aligned}
\end{align}
where $C_0>0$ takes in account the locally Lipschitz property of $h_x$ \eqref{eq:hx_Lip_time} and $\sup_{y\in\cO} \e^{T\mu_x(y)}$, and $C_1>0$ is a suitable constant collecting all terms independent of $\eps$ (notice that $\rho\le T$, $\P$-a.s.). The second term of \eqref{eq:v_x-diff} is instead split as follows
\begin{align}
 \begin{aligned}
 \E\Big[&\e^{-\int_0^\rho \lambda(Y^y_s)\ud s} \big(v_x(t+\rho,Y^y_\rho)-v_x(t-\eps+\rho,Y^y_\rho)\big)\Big]\\
 &=\E\Big[\ind_{\{\sigma_*\leq\tau_d\wedge S \}}\e^{-\int_0^{\sigma_*} \lambda(Y^y_s)\ud s} \big(v_x(t+\sigma_*,Y^y_{\sigma_*})-v_x(t-\eps+\sigma_*,Y^y_{\sigma_*})\big)\Big] \\
 &\hspace{12pt}+\E\Big[\ind_{\{\tau_d\le\sigma_*\wedge S \}}\e^{-\int_0^{\tau_d} \lambda(Y^y_s)\ud s} \big(v_x(t+\tau_d,Y^y_{\tau_d})-v_x(t-\eps+\tau_d,Y^y_{\tau_d})\big)\Big]\\
 &\hspace{12pt}+\E\Big[\ind_{\{S<\sigma_*\wedge \tau_d \}}\e^{-\int_0^S \lambda(Y^y_s)\ud s} \big(v_x(t+S,Y^y_{S})-v_x(t-\eps+S,Y^y_{S})\big)\Big].
\end{aligned}
\end{align}
Then, using that $v_x(t+\sigma_*,Y^y_{\sigma_*})=\alpha_0\geq v_x(t-\eps+\sigma_*,Y^y_{\sigma_*})$, we obtain
\begin{align}\label{eq:v_x-diff@rho}
 \begin{aligned}
 \E\Big[&\e^{-\int_0^\rho \lambda(Y^y_s)\ud s} \big(v_x(t+\rho,Y^y_\rho)-v_x(t-\eps+\rho,Y^y_\rho)\big)\Big]\\
 &\geq\E\Big[\ind_{\{\tau_d\le\sigma_*\wedge S \}}\e^{-\int_0^{\tau_d} \lambda(Y^y_s)\ud s} \big(v_x(t+\tau_d,Y^y_{\tau_d})-v_x(t-\eps+\tau_d,Y^y_{\tau_d})\big)\Big]\\
 &\hspace{12pt}+\E\Big[\ind_{\{S<\sigma_*\wedge \tau_d \}}\e^{-\int_0^S \lambda(Y^y_s)\ud s} \big(v_x(t+S,Y^y_{S})-v_x(t-\eps+S,Y^y_{S})\big)\Big],
\end{aligned}
\end{align}
For the first term on the right-hand side of \eqref{eq:v_x-diff@S}, by the Mean Value theorem, we have that
 \begin{align*}
 \ind_{\{\tau_d\le\sigma_*\wedge S \}}\e^{-\int_0^{\tau_d} \lambda(Y^y_s)\ud s}& \big(v_x(t+\tau_d,Y^y_{\tau_d})-v_x(t-\eps+\tau_d,Y^y_{\tau_d})\big)\\
 &\geq -\ind_{\{\tau_d\le\sigma_*\wedge S \}}C_2 \big|v_{tx}(t+\tau_d-\xi,Y^y_{\tau_d})\big|\eps,
 \end{align*}
 where $C_2\coloneqq \sup_{y\in\cO}\e^{T\mu_x(y)}$ (recall \eqref{eq:lambda} and that $\mu_x$ is bounded from above by Assumption \ref{ass:1}). Then, by continuity of $Y$ and $d$, we obtain
 \begin{equation}\label{eq:v_x-diff@tau_d}
 \ind_{\{\tau_d\le\sigma_*\wedge S \}}\e^{-\int_0^{\tau_d} \lambda(Y^y_s)\ud s} \big(v_x(t+\tau_d,Y^y_{\tau_d})-v_x(t-\eps+\tau_d,Y^y_{\tau_d})\big)\geq -\ind_{\{\tau_d\le\sigma_*\wedge S \}} C_3\eps,
 \end{equation}
 where $C_3\coloneqq C_2\sup_{s\in[0,S]}\sup_{\xi\in[0,t_1]}|v_{tx}(t\!+\!s\!-\!\xi,d(t\!+\!s))|<\infty$, for $t_1\in(\eps,t)$ fixed, by classical interior regularity results for solutions of PDEs (see \cite[Th.\ 3.10]{friedman2008partial}). Indeed, \cite[Th.\ 3.10]{friedman2008partial} implies the continuity of $v_{tx}$ inside the open domain $\cC\cap\cI$. For the second term on the right-hand side of \eqref{eq:v_x-diff@S}, by introducing the probability density function $z\mapsto p(z;S,y)$ of the random variable $Y^y_S$, we have that
 \begin{align}\label{eq:v_x-diff@S}
\begin{aligned} &\E\Big[\ind_{\{S<\sigma_*\wedge \tau_d \}}\e^{-\int_0^S \lambda(Y^y_s)\ud s} \big(v_x(t+S,Y^y_{S})-v_x(t-\eps+S,Y^y_{S})\big)\Big] \\
&\geq-C_2\E\Big[\ind_{\{S<\sigma_*\wedge \tau_d \}}\big(v_x(t-\eps+S,Y^y_{S})-v_x(t+S,Y^y_{S})\big)\Big]\\
&\geq -C_2\int_{d(t+s)}^{b(t+s)} \big(v_x(t-\eps+S,z)-v_x(t+S,z) \big)p(z;S,y)\ud z\\
&=-C_2\Big(\big[(v(t-\eps+S,z)-v(t+S,z))p(z;S,y) \big]^{z=b(t+s)}_{z=d(t+s)}\\
&\hspace{30pt}+\int_{d(t+s)}^{b(t+s)}\big(v(t+S,z)-v(t-\eps+S,z) \big)p'(z;S,y)\ud z\Big)\\
&\geq -C_2 \Big( 1+\int_{d(t+s)}^{b(t+s)}\sup_{\xi\in[0,\eps]}|v_t(t+S-\xi,z)|p'(z;S,y)\ud z\Big)\eps\\
&\geq -C_4\Big( 1+\int_{d(t+s)}^{b(t+s)}D_1(1+|z|^{2\vee p})p'(z;S,y)\ud z\Big)\eps\eqqcolon-C_5\, \eps,
\end{aligned}
\end{align}
where we have used integration by parts in the equality; we have used that $z\mapsto p'(z;S,y)$ is continuous by Assumption \ref{ass:Ypdf} in the third and fourth inequality where $C_4>0$ is a constant depending on $p$ and $v$; the last inequality follows by \eqref{eq:M_time} and letting $C_5>0$ a sufficiently large constant. By using \eqref{eq:h_x-diff}, \eqref{eq:v_x-diff@rho}, \eqref{eq:v_x-diff@tau_d} and \eqref{eq:v_x-diff@S} in \eqref{eq:v_x-diff}, we obtain
 $$v_x(t,y)-v_x(t-\eps,y)\geq -D_4 (1+|y|^p)\eps,$$
 for some constant $D_4>0$. Recalling that $v_x(t,y)-v_x(t-\eps,y)\leq 0$ by Proposition \ref{prop:monotonicity}, we reach the desired result $|v_x(t,y)-v_x(t-\eps,y)|\leq D_4 (1+|y|^p)\eps.$
\end{proof}

\begin{theorem}\label{thm:v_txCont}
The function $v_{tx}$ is continuous at the boundary $\partial \cI$, i.e., for all $(t_0,y_0)\in\partial\cI$ with $t_0<T$ and for any sequence $((t_n,y_n))_{n\in\N}\subseteq\cI$ such that $(t_n,y_n)\to(t_0,y_0)$ as $n\to\infty$, we have
\begin{align*}
\lim_{n\to\infty}v_{tx}(t_n,y_n)=0.
\end{align*}
\end{theorem}

\begin{proof}
 This proof is very similar to the one of Proposition \ref{prop:v_x-Lip-time} and so we refer to that proof for a detailed presentation.

 Fix $(t,y)\in\cC\cap\cI$, so that $a(t)<y<b(t)$. We consider the same function $d:[0,T]\to\R$ as in Proposition \ref{prop:v_x-Lip-time} which is continuous, non-decreasing and such that $a(s)<d(s)<b(s)$ for every $s\in[0,T]$ and such that $y>d(t)$.
 
 Fix $\eps\in(0,t)$. Since $a(t-\eps+s)<d(t-\eps+s)\leq d(t+s)$ for every $s\in[0, T-t)$, we have that
 $$\tau_d(t,y)\leq \tau_d(t-\eps,y)\leq\tau_a(t-\eps,y), \quad \P\text{-a.s.}$$
 Similarly to \eqref{eq:v_x-diff}, letting $\rho=\sigma_*(t,y)\wedge\tau_d(t,y)\wedge S$ for a fixed $S\in(0,T-t)$, we obtain
 \begin{align}\label{eq:v_x-diff1}
 \begin{aligned}
 v_x(t,y)-v_x(t-\eps,y) &\geq \E\Big[\int_0^\rho \e^{-\int_0^s\lambda(Y^y_u)\ud u}\big(h_x(t+s,Y^y_s)-h_x(t-\eps+s,Y^y_s)\big)\ud s\Big] \\
 &\hspace{11pt}+\E\Big[\e^{-\int_0^\rho \lambda(Y^y_s)\ud s} \big(v_x(t+\rho,Y^y_\rho)-v_x(t-\eps+\rho,Y^y_\rho)\big)\Big],
\end{aligned}
 \end{align}
 where we also recall the definition of $\tau_d$ from \eqref{eq:tau_d}.
 
 For the first term on the right-hand side of \eqref{eq:v_x-diff1}, we refine the estimate in \eqref{eq:h_x-diff} using that $\rho\le\sigma_*$, and we apply H\"older's inequality
\begin{equation}\label{eq:h_x-diff1}
\E\Big[\int_0^\rho \e^{-\int_0^s\lambda(Y^y_u)\ud u}\big(h_x(t+s,Y^y_s)-h_x(t-\eps+s,Y^y_s)\big)\ud s\Big]\geq -\eps\,C_1\sqrt{(1+ |y|^{2p})}\sqrt{\E[\sigma_*^2]},
\end{equation}
where $p\in[1,\infty)$, $C_1>0$ are constants derived from the (local) Lipschitz continuity of $t\mapsto h_x(t,y)$ (recall \eqref{eq:hx_Lip_time}) and we have used standard estimates for SDEs. For the second term of \eqref{eq:v_x-diff1}, in the same way as for \eqref{eq:v_x-diff@rho}, we obtain
\begin{align}\label{eq:v_x-diff@rho1}
\begin{aligned}
&\E\Big[\e^{-\int_0^\rho \lambda(Y^y_s)\ud s} \big(v_x(t+\rho,Y^y_\rho)-v_x(t-\eps+\rho,Y^y_\rho)\big)\Big]\\
&\ge\E\Big[\ind_{\{\tau_d\le\sigma_*\wedge S \}}\e^{-\int_0^{\tau_d} \lambda(Y^y_s)\ud s} \big(v_x(t+\tau_d,Y^y_{\tau_d})-v_x(t-\eps+\tau_d,Y^y_{\tau_d})\big)\Big]\\
&\hspace{12pt}+\E\Big[\ind_{\{S<\sigma_*\wedge \tau_d \}}\e^{-\int_0^S \lambda(Y^y_s)\ud s} \big(v_x(t+S,Y^y_{S})-v_x(t-\eps+S,Y^y_{S})\big)\Big].
 \end{aligned}
 \end{align}
 For the first term on the right-hand side of \eqref{eq:v_x-diff@rho1}, in the same way as for \eqref{eq:v_x-diff@tau_d}, we obtain
 \begin{equation}\label{eq:v_x-diff@tau_d1}
 \ind_{\{\tau_d\le\sigma_*\wedge S \}}\e^{-\int_0^{\tau_d} \lambda(Y^y_s)\ud s} \big(v_x(t+\tau_d,Y^y_{\tau_d})-v_x(t-\eps+\tau_d,Y^y_{\tau_d})\big)\geq -\ind_{\{\tau_d\le\sigma_*\wedge S \}} C_2\eps,
 \end{equation}
 for some $C_2>0$. For the second term on the right-hand side of \eqref{eq:v_x-diff@rho1}, we have that
 \begin{align}\label{eq:v_x-diff@S1}
 \begin{aligned}
 \E&\Big[\ind_{\{S<\sigma_*\wedge \tau_d \}}\e^{-\int_0^S \lambda(Y^y_s)\ud s} \big(v_x(t+S,Y^y_{S})-v_x(t-\eps+S,Y^y_{S})\big)\Big]\\
 &\geq - D_4\, \eps\,\E\Big[\ind_{\{S<\sigma_*\wedge \tau_d \}}\big(1+\sup_{s\in[0,S]}|Y^y_S|^p\big) \Big]\geq -\eps\, C_3 \sqrt{\P(S<\sigma_*\wedge\sigma_d)},
 \end{aligned}
 \end{align}
 for some $C_3>0$, where the first inequality is justified by Proposition \ref{prop:v_x-Lip-time} and we have used standard estimates for SDEs and H\"older's inequality in the second inequality. Plugging \eqref{eq:h_x-diff1}, \eqref{eq:v_x-diff@rho1}, \eqref{eq:v_x-diff@tau_d1} and \eqref{eq:v_x-diff@S1} into \eqref{eq:v_x-diff1}, yields
\begin{align*}
v_x(t,y)\!-\!v_x(t\!-\!\eps,y) \geq -\eps\,\bar K \Big(\P(\tau_d<\sigma_*\wedge S) + \sqrt{\P(S<\sigma_*\wedge\tau_d)}+\sqrt{\E[\sigma_*^2]}\Big),
\end{align*}
for a given constant $\bar K>0$. Dividing by $\eps$ and letting $\eps\to 0$, we obtain
\begin{equation}\label{eq:v_tx>}
v_{tx}(t,y)\geq -\bar K \Big( \P(\tau_d<\sigma_*\wedge S) + \sqrt{\P(S<\sigma_*\wedge\tau_d)}+\sqrt{\E[\sigma_*^2]}\Big), \qquad \forall \: (t,y)\in\cC\cap\cI.
\end{equation}
Now, let $(t_0,y_0)\in\partial\cI$ with $t_0<T$ and consider a sequence $((t_n,y_n))_{n\in\N}\subseteq \cI$ such that $(t_n,y_n)\to (t_0,y_0)$ as $n\to\infty$. Since the boundaries $a$ and $b$ are separated, we can take $((t_n,y_n))_{n\in\N}\subseteq \cC\cap\cI$ without loss of generality. Applying \eqref{eq:v_tx>} at $(t_n,y_n)$ for every $n\in\N$, we have
$$\liminf_{n\to\infty} v_{tx}(t_n,y_n)\geq 0,$$
since $\sigma_*(t_n,y_n)\to \sigma_*(t_0,y_0)=0$, $\P$-a.s., as $n\to\infty$ by Proposition \ref{prop:sigma_n-to-sigma_*}. Combining this with $\tau_d>0$, $\P$-a.s. and the fact that $v_{tx}(t_n,y_n)\leq 0$ for every $n\in\N$, by the monotonicity property of $v_x$ from Proposition \ref{prop:monotonicity}, we reach the desired result.
\end{proof}

Notice that, as a simple corollary of Theorem \ref{thm:v_txCont}, we also obtain $v_t$ continuous on $\cM$ and so, in particular, across $\partial \cI$.

\begin{corollary}\label{cor:v_t-cont-M}
The function $v_t$ is continuous on $\cM$.
\end{corollary}

\begin{proof}
Fix $(t,x)\in\cM$, then $t<T$ and $x\geq b(t)$. Let $c\in(a(t),b(t))$ justified by $a(t)<b(t)$ (cf.\ \eqref{eq:a<b}). Then, $(t,c)\in\cC\cap\cI$ and
\begin{align*}
 v(t,x)=\int_{c}^{x}v_x(t,y)\ud y+v(t,c).
\end{align*}
Therefore, the function $v$ is differentiable at $(t,x)$ with respect to $t$ because $v$ is differentiable at $(t,c)\in \cC\cap \cI$ and $v_x$ is continuously differentiable with respect to $t$ by Theorem \ref{thm:v_txCont}. Consequently,
\begin{align*}
 v_t(t,x)=\int_{c}^{x}v_{tx}(t,y)\ud y+v_t(t,c),
\end{align*}
where we differentiate inside the integral by Liebniz rule since $v_x$ and $v_{tx}$ are continuous.
\end{proof}

We now prove that $v_{xx}$ is continuous across the boundary $b$. We recall that $v_{xx}(t,y)=0$ for every $(t,y)\in \text{int}(\cM)$ because $v_x(t,y)=\alpha_0$ for every $(t,y)\in\cM$. Moreover, by the PDE in \eqref{eq:PDE}, we have that $v_{xx}$ is continuous inside $\cC\cap\cI$ and this continuity can be extended to $\overline{\cC\cap\cI}$, because of the continuity inside $\overline{\cC\cap\cI}$ of all other terms involved in \eqref{eq:PDE}.

\begin{theorem}\label{thm:v_xxCont}
The function $v_{xx}$ is continuous across the boundary $\partial\cI$, i.e., for every $(t_0,y_0)\in\partial\cI$ with $t_0<T$ and for any sequence $((t_n,y_n))_{n\in\N}\subseteq\cI$ such that $(t_n,y_n)\to(t_0,y_0)$ as $n\to\infty$, we have
$$\lim_{n\to\infty}v_{xx}(t_n,y_n)=0.$$
\end{theorem}
\begin{proof}
Let $(t,y)\in\cC\cap\cI$, $S\in(0,T-t)$ and $\eps>0$. Recall the definitions of $\tau_a=\tau_a(t,y)$ and $\sigma_*=\sigma_*(t,y)$ in \eqref{eq:taua} and \eqref{eq:sigma*}, respectively, and denote $\bar\tau = \bar\tau(t,y)\coloneqq \sigma_*\wedge\tau_a\wedge S$.

Notice that $(t+s,Y^y_s)\in\cC\cap\cI$ for every $s\in [0,\bar\tau)$. Thus, by applying Dynkin's formula to
$$\e^{-\int_0^s\lambda(Y^{y+\eps}_u)\ud u} v_x(t+s,Y^{y}_s)-\e^{-\int_0^s\lambda(Y^{y}_u)\ud u} v_x(t+s,Y^{y+\eps}_s)$$
on the random interval $[0,\bar\tau]$ and using \eqref{eq:ineqPDE_vx} and \eqref{eq:PDE_vx}, we have that
\begin{align*}
 v_x(t,y+\eps)-v_x(t,y)&\leq\E\Big[\e^{-\int_0^{\bar{\tau}}\lambda(Y^{y+\eps}_u)\ud u}v_x(t\!+\!\bar{\tau},Y^{y+\eps}_{\bar{\tau}})\!-\e^{-\int_0^{\bar{\tau}}\lambda(Y^{y}_u)\ud u}v_x(t\!+\!\bar{\tau},Y^y_{\bar{\tau}})\Big]\\
&\hspace{11pt}+\E\Big[\!\int_{0}^{\bar{\tau}}\!\!\Big(\e^{-\int_0^s\lambda(Y^{y+\eps}_u)\ud u}h_x(t\!+\!s,Y^{y+\eps}_s)\!-\e^{-\int_0^s\lambda(Y^{y}_u)\ud u}h_x(t\!+\!s,Y^y_s)\Big)\,\ud s\Big].
\end{align*}
Then, adding and subtracting $\e^{\int_0^{\bar\tau}\lambda(Y^{y+\eps}_u)\ud u}v_x(t+\bar\tau,Y^y_{\bar\tau})$ and recalling that $y\mapsto v(t,y)$, $y\mapsto v_x(t,y)$, $y\mapsto h(t,y)$, $y\mapsto h_x(t,y)$ and $y\mapsto\lambda(y)$ are all non-decreasing, we obtain
\begin{align}\label{eq:v_xx-add-subtr}
\begin{aligned}
v_x(t,y+\eps)-v_x(t,y)&\leq C_2 \E\big[v_x(t+\bar\tau,Y^{y+\eps}_{\bar\tau})-v_x(t+\bar\tau,Y^{y}_{\bar\tau}) \big]\\
&\hspace{11pt}+\E\Big[\big(\e^{\int_0^{\bar\tau}\lambda(Y^{y+\eps}_u)\ud u}-\e^{\int_0^{\bar\tau}\lambda(Y^{y}_u)\ud u}\big)v_x(t+\bar\tau,Y^y_{\bar\tau})\Big]\\
&\hspace{11pt}+C_2 \E\Big[\int_0^{\bar\tau}\big(h_x(t+s,Y^{y+\eps}_{s})-h_x(t+s,Y^{y}_{s}) \big)\ud s\Big]\\
&\hspace{11pt}+\E\Big[\int_0^{\bar\tau}\big(\e^{\int_0^{s}\lambda(Y^{y+\eps}_u)\ud u}-\e^{\int_0^{s}\lambda(Y^{y}_u)\ud u}\big)h_x(t+s,Y^y_{s})\ud s\Big],
 \end{aligned}
\end{align}
where $C_2\coloneqq \sup_{y\in\cO}\e^{T\mu_x(y)}$. For the first term on the right-hand side of \eqref{eq:v_xx-add-subtr}, we have that
\begin{align*}
\begin{aligned}
\E\big[v_x(t+\bar\tau,Y^{y+\eps}_{\bar\tau})-v_x(t+\bar\tau,Y^{y}_{\bar\tau}) \big] &= \E\big[\ind_{\{\sigma_*\leq \tau_a\wedge S\}}(v_x(t+\sigma_*,Y^{y+\eps}_{\sigma_*})-v_x(t+\sigma_*,Y^{y}_{\sigma_*})) \big]\\
&\hspace{11pt}+\E\big[\ind_{\{\tau_a\wedge S<\sigma_*\}}(v_x(t+\bar\tau,Y^{y+\eps}_{\bar\tau})-v_x(t+\bar\tau,Y^{y}_{\bar\tau})) \big]\\
&=\E\big[\ind_{\{\tau_a\wedge S<\sigma_*\}}(v_x(t+\bar\tau,Y^{y+\eps}_{\bar\tau})-v_x(t+\bar\tau,Y^{y}_{\bar\tau})) \big],
\end{aligned}
\end{align*}
where the last equality comes from the fact that
$$0\leq v_x(t+\sigma_*,Y^{y+\eps}_{\sigma_*})-v_x(t+\sigma_*,Y^{y}_{\sigma_*})\leq \alpha_0-\alpha_0=0.$$
Therefore, for the first term on the right-hand side of \eqref{eq:v_xx-add-subtr}, we obtain
\begin{align}\label{eq:v_xxFirst}
\begin{aligned}
&\E\big[v_x(t\!+\!\bar\tau,Y^{y+\eps}_{\bar\tau})\!-\!v_x(t\!+\!\bar\tau,Y^{y}_{\bar\tau}) \big]\\
&=\E\Big[\ind_{\{\tau_a\wedge S<\sigma_*\}}(v_x(t\!+\!\bar\tau,Y^{y+\eps}_{\bar\tau})\!-\!v_x(t\!+\!\bar\tau,Y^{y}_{\bar\tau})) \Big]\\
&=\E\Big[\ind_{\{\tau_a\wedge S<\sigma_*\}}\ind_{\{Y^{y+\eps}_{\bar\tau}\geq b(t+\bar\tau)\}}(v_x(t\!+\!\bar\tau,Y^{y+\eps}_{\bar\tau})\!-\!v_x(t\!+\!\bar\tau,Y^{y}_{\bar\tau})) \Big]\\
&\hspace{11pt}+\!\E\Big[\ind_{\{\tau_a\wedge S<\sigma_*\}} \ind_{\{Y^{y+\eps}_{\bar\tau}< b(t+\bar\tau)\}}(v_x(t\!+\!\bar\tau,Y^{y+\eps}_{\bar\tau})\!-\!v_x(t\!+\!\bar\tau,Y^{y}_{\bar\tau}))\Big]\\
&\leq C_3\E\big[\ind_{\{\tau_a\wedge S<\sigma_*\}} \ind_{\{Y^{y+\eps}_{\bar\tau}\geq b(t+\bar\tau)\}}\big|b(t+\bar\tau)\!-\!Y^y_{\bar\tau}\big|\, \big]\\
&\hspace{11pt}+\!C_3\E\big[\ind_{\{\tau_a\wedge S<\sigma_*\}} \ind_{\{Y^{y+\eps}_{\bar\tau}< b(t+\bar\tau)\}}\big|Y^{y+\eps}_{\bar\tau}\!-\!Y^y_{\bar\tau}\big|\, \big]\\
&\leq C_3 \E\big[\ind_{\{\tau_a\wedge S<\sigma_*\}}\big|Y^{y+\eps}_{\bar\tau}\!-\!Y^y_{\bar\tau}\big|\, \big]\\
 &\leq K_1\eps \sqrt{\P(\tau_a\wedge S<\sigma_*)},
 \end{aligned}
\end{align}
where in the first inequality we have used the Mean value theorem with 
$$C_3\coloneqq \sup_{s\in[0,S]}\sup_{\xi\in[a(t+s),b(t+s)]}v_{xx}(t+s,\xi)<\infty$$
and in the last inequality we have used H\"older's inequality with $K_1>0$ derived from standard estimates for SDEs. For the second term on the right-hand side of \eqref{eq:v_xx-add-subtr}, we have that
\begin{align}\label{eq:v_xxSecond}
\begin{aligned}
 &\E\Big[\big(\e^{\int_0^{\bar\tau}\lambda(Y^{y+\eps}_u)\ud u}-\e^{\int_0^{\bar\tau}\lambda(Y^{y}_u)\ud u}\big)v_x(t+\bar\tau,Y^y_{\bar\tau})\Big]\\
 &\leq C_2\E\Big[\int_0^{\bar\tau} |\mu_x(Y^{y+\eps}_u)-\mu_x(Y^y_u)|\ud u\Big]\\
 &\le C_2 \E\Big[C\sup_{u\in[0,T]}\Big(1+|Y^{y+\eps}_u|^p+|Y^{y}_u|^p\Big)|Y_u^{y+\eps}-Y^{y}_u|\bar\tau\Big]\\
 &\le C_2C \E\Big[\sup_{u\in[0,T]}\Big(1+|Y^{y+\eps}_u|^p+|Y^{y}_u|^p\Big)^4\Big]^{1/4}\E\Big[\sup_{u\in[0,T]}|Y_u^{y+\eps}-Y^{y}_u|^4\Big]^{1/4}\sqrt{\E\big[\bar\tau^2\big]}\\
 &\leq C_4 \eps \sqrt{\E[\sigma_*^2]},
 \end{aligned}
\end{align}
where we used that $v_x$ is bounded in the first inequality; we used the local Lipschitz property of $\mu_x$ \eqref{eq:mu_x-lip} and took the supremum in the second inequality; the third inequality is a consequence of H\"older's inequality; finally, we collect $C_2 C$ and constant from standard estimates for SDEs in $C_4$. For the third term on the right-hand side of \eqref{eq:v_xx-add-subtr}, we obtain
\begin{equation}\label{eq:v_xxThird}
\E\Big[\int_0^{\bar\tau}\big(h_x(t+s,Y^{y+\eps}_{s})-h_x(t+s,Y^{y}_{s}) \big)\ud s\Big]\leq\sup_{s\in[0,S]} L(t+s) K_2\eps\sqrt{\E[\sigma^2_*]},
\end{equation}
where we have used H\"older's inequality with $t\mapsto L(t)$ deriving from assumption \eqref{eq:h_x-lip} and $K_2>0$ deriving from standard estimates for SDEs. Similarly to \eqref{eq:v_xxThird}, for the fourth term on the right-hand side of \eqref{eq:v_xx-add-subtr}, we have that
\begin{equation}\label{eq:v_xxFourth}
\E\Big[\int_0^{\bar\tau}\big(\e^{\int_0^{s}\lambda(Y^{y+\eps}_u)\ud u}-\e^{\int_0^{s}\lambda(Y^{y}_u)\ud u}\big)h_x(t+s,Y^y_{s})\ud s\Big]\leq C_5 \eps \sqrt{\E[\sigma_*^2]},
\end{equation}
where $C_5>0$ is a sufficiently large constant depending of $ C_2$, $\sup_{s\in [t,t+S]}\sup_{\xi\in[a(t+s),b(t+s)]}h_x(s,\xi)<\infty$ and the upper bound on the exponential (justified by similar calculations to those in \eqref{eq:v_xxSecond}). Plugging \eqref{eq:v_xxFirst}, \eqref{eq:v_xxSecond}, \eqref{eq:v_xxThird} and \eqref{eq:v_xxFourth} into \eqref{eq:v_xx-add-subtr}, we obtain
$$v_x(t,y+\eps)-v_x(t,y)\leq \tilde K\eps \big(\sqrt{\P(\tau_a\wedge S<\sigma_*)}+\sqrt{\E[\sigma^2_*]} \big),$$
for some constant $\tilde K>0$. Since $y\mapsto v_x(t,y)$ is non-decreasing, in fact we have
$$|v_x(t,y+\eps)-v_x(t,y)|\leq \tilde K\eps \big(\sqrt{\P(\tau_a\wedge S<\sigma_*)}+\sqrt{\E[\sigma^2_*]} \big).$$
Dividing by $\eps$ and letting $\eps\to 0$, we obtain
\begin{equation}\label{eq:v_xxBounds}
 |v_{xx}(t,y)|\leq \tilde K\eps \big(\sqrt{\P(\tau_a\wedge S<\sigma_*)}+\sqrt{\E[\sigma^2_*]} \big), \qquad \forall \: (t,y)\in\cC\cap\cI.
\end{equation}
Now let $(t_0,y_0)\in\partial\cI$ with $t_0<T$ and $((t_n,y_n))_{n\in\N}\subseteq\cI$ such that $(t_n,y_n)\to(t_0,x_0)$ as $n\to\infty$. Without loss of generality, since $a$ and $b$ are separated, we can consider $((t_n,y_n))_{n\in\N}\subseteq\cC\cap\cI$. Therefore, equation \eqref{eq:v_xxBounds} holds for $v_{xx}(t_n,y_n)$ for every $n\in\N$ and letting $n\to\infty$ we obtain the desired result, since $\sigma_*(t_n,y_n)\to\sigma_*(t_0,y_0)= 0$ as $n\to\infty$ by Proposition \ref{prop:sigma_n-to-sigma_*}.
\end{proof}

\begin{remark}\label{rmk:v_xx-disc}
 Notice that, instead, we expect potential discontinuities of $v_{xx}$ across the boundary $\partial\cC$. Indeed, for any $t\in [0,T)$ and any $\tau\in\cT_t$, we have that
 \begin{align*}
 v(t,x)&\geq \E_x\Big[\e^{-r\tau}g(t+\tau)+\int_0^\tau \e^{-rs}h(t+s,X^{\nu^*}_s)\ud s+\int_{[0,\tau] }\e^{-rs}\alpha_0\ud|\nu^*|_s \Big]\\
 &= g(t,x) + \E_x\Big[ \int_0^\tau \e^{-rs}\Theta(t+s,X^{\nu^*}_s)\ud s+\int_{[0,\tau] }\e^{-rs}\alpha_0\ud|\nu^*|_s\Big],
 \end{align*}
 where $\Theta$ was defined in \eqref{eq:Theta} and in the equality we have used It\^o's formula. Since $\cC=\{(t,x)\in[0,T)\times\cO|v(t,x)>g(x) \}$, we obtain $\{(t,x)\in[0,T)\times\cO|\Theta(t,x)>0 \}\subseteq \cC$ or, equivalently, $\cS\subseteq \{(t,x)\in[0,T)\times\cO|\Theta(t,x)\leq0 \}$. Therefore, $\Theta(t_0,x_0)\leq 0$ for every $(t_0,x_0)\in\cS$ with $t_0<T$. Now let $((t_n,x_n))_{n\in\N}\subseteq \cC$ such that $(t_n,x_n)\to (t_0,x_0)\in\partial\cC$ as $n\to\infty$, with $t_0<T$. By continuity of $v$, $v_t$ and $v_x$ across $\partial\cC$ and by the PDE in \eqref{eq:PDE}, we obtain
 \begin{align}\label{eq:pdevxx}
 \begin{aligned}
 \lim_{n\to\infty} v_{xx}(t_n,x_n)&=- \tfrac{2}{\sigma^2(x_0)}\lim_{n\to\infty}\big[v_t(t_n,x_n)-rv(t_n,x_n)+h(t_n,x_n)+\mu(x_n) v_x(t_n,x_n) \big]\\
 &=- \tfrac{2}{\sigma^2(x_0)}\Theta(t_0,x_0)\geq 0.
 \end{aligned}
 \end{align}
 On the other hand, $v_{xx}(t,x)=0$ for every $(t,x)\in \inter(\cS)$.
\end{remark}

\section*{Acknowledgements}
We would like to thank Tiziano De Angelis for the fruitful discussions on this work.

\bibliographystyle{plain}
\bibliography{Bibliography}

\begin{thebibliography}{10}

\bibitem{baldi2017stochastic}
P.~Baldi.
\newblock {\em Stochastic calculus}.
\newblock Springer, 2017.

\bibitem{bayraktar2013controller}
E.~Bayraktar and Y.-J. Huang.
\newblock On the multidimensional controller-and-stopper games.
\newblock {\em {SIAM} J.\ Control Optim.}, 51(2):1263--1297, 2013.

\bibitem{bayraktar2011regularity}
E.~Bayraktar and V.R. Young.
\newblock Proving regularity of the minimal probability of ruin via a game of
  stopping and control.
\newblock {\em Finance Stoch.}, 15(4):785--818, 2011.

\bibitem{billingsley2013convergence}
P.~Billingsley.
\newblock {\em Convergence of probability measures}.
\newblock Wiley Series in Probability and Statistics. John Wiley \& Sons,
  second edition, 1999.

\bibitem{bovo2024halfline}
A.~Bovo and T.~De~Angelis.
\newblock Finite-time horizon, stopper vs. singular-controller games on the
  half-line.
\newblock {\em arXiv:2409.06049}, 2024.

\bibitem{bovo2024saddle}
A.~Bovo and T.~De~Angelis.
\newblock On the saddle point of a zero-sum stopper vs.\ singular-controller
  game.
\newblock {\em Stochastic Process.\ Appl.}, 182, 2025.

\bibitem{bovo2024variational}
A.~Bovo, T.~De~Angelis, and E.~Issoglio.
\newblock Variational inequalities on unbounded domains for zero-sum
  singular-controller vs. stopper games.
\newblock {\em Math.\ Oper.\ Res.}, 50(1):277--312, 2025.

\bibitem{bovo2023}
A.~Bovo, T.~De~Angelis, and J.~Palczewski.
\newblock Zero-sum stopper vs. singular-controller games with constrained
  control directions.
\newblock {\em {SIAM} J.\ Control Optim.}, 62(4):2203--2228, 2024.

\bibitem{bovo2023b}
A.~Bovo, T.~De~Angelis, and J.~Palczewski.
\newblock Stopper vs.\ singular-controller games with degenerate diffusions.
\newblock {\em Appl.\ Math.\ Optim.}, 91(Article 3), 2025.

\bibitem{choukroun2015bsde}
S.~Choukroun, A.~Cosso, and H.~Pham.
\newblock Reflected {BSDEs} with nonpositive jumps, and controller-and-stopper
  games.
\newblock {\em Stochastic Process.\ Appl.}, 125(2):597--633, 2015.

\bibitem{cox2015embedding}
A.M.G. Cox and G.~Peskir.
\newblock Embedding laws in diffusions by functions of time.
\newblock {\em Ann.\ Probab.}, 43(5):2481--2510, 2015.

\bibitem{deangelis2021assets}
T.~De~Angelis, F.~Gensbittel, and S.~Villeneuve.
\newblock A {D}ynkin game on assets with incomplete information on the return.
\newblock {\em Math.\ Oper.\ Res.}, 46(1):28--60, 2021.

\bibitem{2022salami}
E.~Ekstr\"om, K.~Lindensj\"o, and M.~Olofsson.
\newblock How to detect a salami slicer: A stochastic controller-and-stopper
  game with unknown competition.
\newblock {\em SIAM J.\ Control Optim.}, 60(1):545--574, 2022.

\bibitem{ekstrom2023finetti}
E.~Ekstr{\"o}m, A.~Milazzo, and M.~Olofsson.
\newblock The de {F}inetti problem with uncertain competition.
\newblock {\em SIAM J.\ Control Optim.}, 61(5):2997--3017, 2023.

\bibitem{fleming2012deterministic}
W.H. Fleming and R.W. Rishel.
\newblock {\em Deterministic and stochastic optimal control}, volume~1 of {\em
  Applications of {M}athematics}.
\newblock Springer New York, NY, 1975.

\bibitem{friedman2008partial}
A.~Friedman.
\newblock {\em Partial differential equations of parabolic type}.
\newblock Dover Publications, Mineola, NY, 2008.

\bibitem{hernandez2015zero}
D.~Hernandez-Hernandez, R.S. Simon, and M.~Zervos.
\newblock A zero-sum game between a singular stochastic controller and a
  discretionary stopper.
\newblock {\em Ann.\ Appl.\ Probab.}, 25(1):46--80, 2015.

\bibitem{hernandez2015zsgsingular}
D.~Hernandez-Hernandez and K.~Yamazaki.
\newblock Games of singular control and stopping driven by spectrally one-sided
  {L}{\'e}vy processes.
\newblock {\em Stochastic Process.\ Appl.}, 125(1):1--38, 2015.

\bibitem{karatzas1984bridge1}
I.~Karatzas and S.E. Shreve.
\newblock Connections between optimal stopping and singular stochastic control
  {I}. {M}onotone follower problems.
\newblock {\em {SIAM} J.\ Control Optim.}, 22(6):856--877, 1984.

\bibitem{karatzas1985bridge2}
I.~Karatzas and S.E. Shreve.
\newblock Connections between optimal stopping and singular stochastic control
  {II}. {R}eflected follower problems.
\newblock {\em {SIAM} J.\ Control Optim.}, 23(3):433--451, 1985.

\bibitem{karatzas1998brownian}
I.~Karatzas and S.E. Shreve.
\newblock {\em Brownian Motion and Stochastic Calculus}, volume 113 of {\em
  Graduate Texts in Mathematics}.
\newblock Springer New York, NY, second edition, 1998.

\bibitem{maitra1996gambler}
A.P. Maitra and W.D. Sudderth.
\newblock The gambler and the stopper.
\newblock In {\em Statistics, probability and game theory}, volume~30 of {\em
  IMS Lecture Notes Monogr.\ Ser.}, pages 191--208. Inst.\ Math.\ Statist.,
  Hayward, CA, 1996.

\bibitem{menozzi2021density}
S.~Menozzi, A.~Pesce, and X.~Zhang.
\newblock Density and gradient estimates for non degenerate brownian sdes with
  unbounded measurable drift.
\newblock {\em J.\ Differential Equations}, 272:330--369, 2021.

\bibitem{peskir2006optimal}
G.~Peskir and A.~Shiryaev.
\newblock {\em Optimal stopping and free-boundary problems}.
\newblock Birkh\"{a}user Verlag, Basel, 2006.

\bibitem{pilipenko2014reflection}
A.~Pilipenko.
\newblock {\em An introduction to stochastic differential equations with
  reflection}, volume~1 of {\em Lectures in {P}ure and {A}pplied
  {M}athematics}.
\newblock Universit\"atsverlag Potsdam, 2014.

\bibitem{protter2005stochastic}
P.E. Protter.
\newblock {\em Stochastic integration and differential equations}, volume~21 of
  {\em Stochastic Modelling and Applied Probability}.
\newblock Springer Berlin, second edition, 2005.

\bibitem{rogers2000diffusions1}
L.C.G Rogers and D.~Williams.
\newblock {\em Diffusions, {M}arkov processes, and {M}artingales: foundations},
  volume~1.
\newblock Cambridge university press, 2000.

\bibitem{rogers2000diffusions}
L.C.G. Rogers and D.~Williams.
\newblock {\em Diffusions, {M}arkov processes, and {M}artingales: {I}t{\^o}
  calculus}, volume~2.
\newblock Cambridge university press, second edition, 2000.

\bibitem{williams1991probability}
D.~Williams.
\newblock {\em Probability with Martingales}.
\newblock Cambridge University Press, 1991.

\end{thebibliography}

\end{document}